\documentclass[11pt,reqno]{amsart}

\usepackage{enumerate, amsmath, amsthm, amsfonts, amssymb, 
mathrsfs, graphicx, paralist, ulem}
\usepackage[usenames, dvipsnames]{color}
\usepackage[margin=1in]{geometry} 
\usepackage{hyperref}
\usepackage{comment}

\usepackage{mathabx} 
\usepackage{breqn} 
\usepackage{tikz-cd}
\usepackage{amsrefs}

\numberwithin{equation}{section}
\newtheorem{Theorem}[equation]{Theorem}
\newtheorem{Proposition}[equation]{Proposition}
\newtheorem{Lemma}[equation]{Lemma}
\newtheorem{Corollary}[equation]{Corollary}
\newtheorem{Conjecture}[equation]{Conjecture}
\newtheorem{Question}[equation]{Question}

\theoremstyle{definition}
\newtheorem{Remark}[equation]{Remark}

\newtheorem{eg}[equation]{Example}
\newenvironment{example}[1][]{\begin{eg}[#1] \pushQED{\qed}}{\popQED \end{eg}}
\newtheorem{Definition}[equation]{Definition}

\newcommand{\cE}{\mathcal{E}}

\newcommand{\cF}{\mathcal{F}}

\newcommand{\cG}{\mathcal{G}}

\newcommand{\cH}{\mathcal{H}}

\newcommand{\cI}{\mathcal{I}}

\newcommand{\cJ}{\mathcal{J}}

\newcommand{\cL}{\mathcal{L}}

\newcommand{\cO}{\mathcal{O}}

\newcommand{\cP}{\mathcal{P}}

\newcommand{\cS}{\mathcal{S}}
\newcommand{\fS}{\mathfrak{S}}

\newcommand{\cT}{\mathcal{T}}

\newcommand{\bm}{\mathbf{m}}

\newcommand{\bn}{\mathbf{n}}

\renewcommand{\AA}{\mathbb{A}}

\newcommand{\CC}{\mathbb{C}}

\newcommand{\FF}{\mathbb{F}}

\newcommand{\NN}{\mathbb{N}}

\newcommand{\PP}{\mathbb{P}}
\newcommand{\QQ}{\mathbb{Q}}

\newcommand{\VV}{\mathbb{V}}

\newcommand{\ZZ}{\mathbb{Z}}

\renewcommand{\phi}{\varphi}
\renewcommand{\emptyset}{\varnothing}
\newcommand{\eps}{\varepsilon}
\newcommand{\injects}{\hookrightarrow}
\newcommand{\surjects}{\twoheadrightarrow}
\renewcommand{\tilde}[1]{\widetilde{#1}}

\newcommand{\leftexp}[2]{\vphantom{#2}^{#1} #2}

\newcommand{\arxiv}[1]{\href{http://arxiv.org/abs/#1}{{\tt arXiv:#1}}}

\DeclareMathOperator{\im}{image}

\renewcommand{\hom}{\operatorname{Hom}}

\DeclareMathOperator{\End}{End}

\DeclareMathOperator{\Sym}{Sym}

\DeclareMathOperator{\vdim}{vdim}

\DeclareMathOperator{\Spec}{Spec}

\DeclareMathOperator{\sgn}{sgn}

\newcommand{\GL}{GL}
\newcommand{\SL}{SL}

\newcommand{\Gr}{\mathbf{Gr}}

\newcommand{\SGr}{\mathbf{SGr}}

\newcommand{\fgl}{\mathfrak{gl}}

\newcommand{\suchthat}{\mid}

\newcommand{\Span}{\mathrm{Span}}
\renewcommand{\Gr}{\cG r}
\renewcommand{\SGr}{\cS\Gr}
\renewcommand{\det}{\mathrm{det}}
\newcommand{\mmap}{\mathrm M}

\begin{document}

\title[The equations defining affine Grassmannians in type A]{The equations defining affine Grassmannians in type A and a conjecture of Kreiman, Lakshmibai, Magyar, and Weyman}
\author{Dinakar Muthiah}
\address{D.~Muthiah:
Department of Mathematics and Statistics, University of Massachusetts Amherst, USA}
\email{muthiah@math.umass.edu}

\author{Alex Weekes}
\address{A.~Weekes: Perimeter Institute for Theoretical Physics, Canada}
\email{aweekes@perimeterinstitute.ca}

\author{Oded Yacobi}
\address{O.~Yacobi: School of Mathematics and Statistics, University of Sydney, Australia}
\email{oded.yacobi@sydney.edu.au}
\maketitle

\begin{abstract}
The affine Grassmannian of $\SL_n$ admits an embedding into the Sato Grassmannian, which further admits a Pl\"ucker embedding into the projectivization of Fermion Fock space.  Kreiman, Lakshmibai, Magyar, and Weyman describe the linear part of the ideal defining this embedding in terms of certain elements of the dual of Fock space called ``shuffles'', and they conjecture that these elements together with the Pl\"ucker relations suffice to cut out the affine Grassmannian.   We give a proof of this conjecture in two steps: first we reinterpret the shuffles equations in terms of Frobenius twists of symmetric functions. Using this, we reduce to a finite-dimensional problem, which we solve.  For the second step we introduce a finite-dimensional analogue of the affine Grassmannian of $\SL_n$, which we conjecture to be precisely the reduced subscheme of a finite-dimensional Grassmannian consisting of subspaces invariant under a nilpotent operator.  
\end{abstract}

\newcommand{\G}{G}
\section{Introduction}

Given an algebraic group $\G$, let $\Gr_\G$ denote the affine Grassmannian, which is a homogeneous space for the loop group of $\G$. We consider $\G=\SL_n$.  In this case, the affine Grassmannian can be realized as a moduli space of certain linear algebraic data called the lattice model of $\Gr_{\SL_n}$.  We also consider the (charge-$0$) Sato Grassmannian $\SGr$, which is the moduli space of subspaces of virtual dimension zero inside an infinite-dimensional vector space \cite[\S 2]{KLMW}, \cite[\S 3.7]{D}. The lattice model of $\Gr_{\SL_n}$ gives rise to a closed embedding:
\begin{align}
  \label{eq:7}
  \Gr_{\SL_n} \hookrightarrow  \SGr
\end{align}
Our goal is to understand the projective geometry of $\Gr_{\SL_n}$ under this embedding.

The Sato Grassmannian can also be realized as an explicit increasing union of finite dimensional Grassmannians. In exact analogy with finite-dimensional Grassmannians, we have a Pl\"ucker embedding
\begin{align}
  \label{eq:29}
 \SGr \hookrightarrow \PP(\cF)
\end{align}
where  $\cF$ is the (charge-$0$) Fermion Fock space.  Both sides are viewed as ind-projective ind-schemes.

Consider the set $\mathfrak{S}_n$ of linear forms on $\PP(\cF)$ that vanish on $\Gr_{\SL_n}$, and let $V(\Lambda_0) \subseteq \cF$ be the subspace of $\cF$ given by the vanishing of $\mathfrak{S}_n$. We note that over a field of characteristic zero, $V(\Lambda_0)$ is identified with the basic representation of affine $\SL_n$ (hence the notation). Therefore we have the following commutative square of closed embeddings:
\begin{equation}
  \label{eq:intro-klmw-diagram}
\begin{tikzcd}
 \Gr_{\SL_n} \arrow[d,hookrightarrow]                    \arrow[r,hookrightarrow] &  \PP(V(\Lambda_0))  \arrow[d,hookrightarrow] \\
 \SGr                  \arrow[r,hookrightarrow] &  \PP(\cF) \
\end{tikzcd}
\end{equation}

Working over an arbitrary field, Kreiman, Lakshmibai, Magyar, and Weyman  \cite{KLMW} explicitly describe the set $\fS_n$ in terms of explicit ``shuffle operators''. Furthermore, they conjecture \cite[\S 4]{KLMW} that these linear forms cut out $\Gr_{\SL_n}$ inside of $\SGr$:
\begin{Conjecture}[The KLMW Conjecture]
\label{conj:KLMW}
The  square \eqref{eq:intro-klmw-diagram} is Cartesian.  
\end{Conjecture}
Our main result is a proof of the KLMW Conjecture (Theorem \ref{thm:klmw-conjecture}) over an arbitrary base ring.  More precisely, we show that the square \eqref{eq:intro-klmw-diagram} is Cartesian in the category of ind-schemes.

\subsection{Some remarks}
This result in particular implies that the closed embedding
\begin{equation}
\Gr_{\SL_n} \hookrightarrow \PP(\cF)
\end{equation} 
is cut out by the ideal generated by quadratic Pl\"ucker equations and linear shuffle equations.  This statement should again be understood in the ind-scheme-theoretic sense: this ideal defines $\Gr_{\SL_n}\hookrightarrow \PP(\cF)$ with its reduced induced scheme structure.  In other words, it defines the corrrect ideal sheaf.  We note however that this ideal may not be saturated, although we conjecture that this is the case (recall that any closed subscheme of projective space is defined by a unique saturated ideal, the largest homogeneous ideal defining the subscheme, see e.g. \cite[Section 13.10]{GW}).

\subsubsection{Relation to KP hierarchy}
  When working over $\CC$, the embedding $\SGr \hookrightarrow \PP(\cF)$ is intimately related to the KP (Kadomtsev-Petviashvili) hierarchy of soliton equations \cite{MJD},\cite[\S 14]{K}.  This idea stems from fundamental work of Sato \cite{S1}, developed by Date-Jimbo-Kashiwara-Miwa \cite{DJKM} as well as Kac-Peterson \cite{KP}.  Moreover, the diagram (\ref{eq:intro-klmw-diagram}) is related to the reduction of the KP hierarchy to the $n$-KdV (Korteweg-de Vries) hierarchy.  Thus, Conjecture \ref{conj:KLMW} can be viewed as a characteristic-free version of this reduction of integrable hierarchies.
  
\subsubsection{Equations defining Kac-Moody partial flag varieties} Work of Kostant, Kac-Peterson, and Kumar (see \cite[Section 10.1]{Ku} for an overview) identifies defining (quadratic) equations of the closed embedding $\cG/\cP \hookrightarrow \PP(V(\Lambda))$, where $\cG$ is a symmetrizable Kac-Moody group, $\cP$ is a parabolic subgroup, and $V(\Lambda)$ is the irreducible representation of highest weight $\Lambda$.  In the special case where $\cG=\widehat{\SL}_n$ and $V(\Lambda_0)$ is the basic representation, one obtains the defining equations for $\Gr_{\SL_n} \hookrightarrow \PP(V(\Lambda_0))$.  

We note that this differs from our setting. The appearance of the Sato Grassmannian in our work is a specifically affine type phenomenon that does not generalize to other Kac-Moody types. Moreover, we doubt that there is a nice generalization beyond type A. Nonetheless, one could imagine using the above results from \cite{Ku} to approach this problem, generalizing for example the discussion in \cite[Section 14.12]{K}. We note that Kumar and Kac work in characteristic zero, whereas we work over any arbitrary ring. In fact, the positive characteristic case is most interesting from our perspective (cf. Section \ref{nilorbits} below).  

\subsection{Our approach}

In addition to our positive answer to the KLMW Conjecture, we think that our approach to the problem is of independent interest.  In particular, it leads to new results (and conjectures) about  schemes of invariant subspaces in finite dimensional Grassmannians.

\subsubsection{Using $\Gr_{\GL_n}$ as an intermediary}
We consider the two-step inclusion
\begin{align}
   \label{eq:62}
   \Gr_{\SL_n} \hookrightarrow \Gr_{\GL_n} \hookrightarrow  \SGr
 \end{align}
and study each step separately.  We study the inclusion $\Gr_{\SL_n} \hookrightarrow \Gr_{\GL_n}$ in \S \ref{sec:sym-functions} by restriction to the big cell, and prove that the vanishing of the shuffle operators cuts out $\Gr_{\SL_n}$ inside of $\Gr_{\GL_n}$ (Theorem \ref{thm:shuffle-equals-det} and Proposition \ref{prop:simplified-klmw-problem}). So we are reduced to showing a positive answer to the following question. 
\begin{Question}
   \label{ques:klmw-shuffle-gr-gln}
 Are the equations defining $\Gr_{\GL_n} \hookrightarrow \SGr$ already imposed by the KLMW shuffle equations?
\end{Question}

\subsubsection{A finite-dimensional analogue} To answer this question, we notice that it admits the following finite dimensional analogue.
Let $V$ be a vector space, let $k$ be an integer with $1 \leq k \leq \dim V$, and let $T : V \rightarrow V$ be a nilpotent operator. 
Analagous to the inclusion $\Gr_{\GL_n} \hookrightarrow  \SGr$, we have the inclusion
\begin{align}
  \label{eq:63}
  \cG^T \hookrightarrow \Gr(k,V) 
\end{align}
where $\Gr(k,V)$ is the Grassmannian of $k$-planes in $V$, and $\cG^T$ is the closed subscheme of $\Gr(k,V)$ consisting of $k$-planes that are invariant under $T$.

We are thus led to ask: what plays the role of $\Gr_{\SL_n}$? The shuffle operators suggest an answer. The KLMW shuffle operators admit a natural analogue for $T$, and we define $\cS^T$ to be the vanishing of the $T$-analogue of shuffle equations inside of $\Gr(k,V)$. The analogue of Question \ref{ques:klmw-shuffle-gr-gln} in this setting is the following:
\begin{Question}
  \label{ques:ess-t-inside-gee-t}
 Is $\cS^T$ a closed subscheme of  $\cG^T$?
\end{Question}
To answer this question we replace the scalar-valued polynomial equations defining Grassmannians and related projective schemes with vector-valued polynomial equations. The source of these vector-valued polynomials are certain  
operators $\Omega_k$ that are constructed from Clifford operators acting on the tensor square of Fermion Fock space (with varying charge). We call these operators KP two-tensors because of their relationship with the Kadomtsev-Petviashvili hierarchy. 

In \S  \ref{sec:kp-two-tensors}, we show that the Pl\"ucker embedding of Grassmmannians and incidence varieties can be succinctly described using KP two-tensors.
In characteristic zero, the Pl\"ucker equations simplify and there is a well-known relationship between the Pl\"ucker equations and the first KP two-tensor $\Omega_1$ \cite[Exercise 14.27]{K}. In Theorem \ref{Theorem:Pluck} we show that this relationship generalizes to arbitrary characteristic if one includes the higher KP two-tensors.

We then give explicit equations for $\cG^T$ by embedding it in an incidence variety (Theorem \ref{thm:equations-for-G-T}), and our formalism of vector-valued polynomials and KP two-tensors gives us a quick positive answer (Theorem \ref{thm:FinAnalog}) to Question \ref{ques:ess-t-inside-gee-t}.

So we have the two-step inclusion
\begin{align}
  \label{eq:65}
  \cS^T \hookrightarrow \cG^T \hookrightarrow \Gr(k,V) 
\end{align}
which is our finite-dimensional analogue of \eqref{eq:62}.
For appropriate choices of $T$ and $V$, \eqref{eq:65} gives the usual finite-dimensional truncation of \eqref{eq:62}. We therefore deduce a positive answer to Question \ref{ques:klmw-shuffle-gr-gln} and prove the KLMW Conjecture (\S \ref{section:proofKLMW}).

\subsection{Further directions}

\subsubsection{Toward standard monomial theory}
In \cite{KLMW} the authors state that one of their goals is to produce a standard monomial theory for the homogenous coordinate ring of $\Gr_{\SL_n}$. Although there exists a general abstract theory of standard monomials for any partial flag variety of a symmetrizable Kac-Moody group \cite{L}, it is hoped that in type A it is possible to define a more explicit standard monomial theory in the style of Hodge.

Our results show that the ideal considered in \cite{KLMW} indeed defines $\Gr_{\SL_n}$ inside $\SGr$. However, we do not know if this ideal is radical (or equivalently saturated). We conjecture that it is radical and therefore provides an explicit presentation for the homogeneous coordinate ring $\Gr_{\SL_n}$ (Conjecture \ref{conj:radical}).  We hope that a good notion of standard monomial theory for the KLMW ideal will give positive answer to this conjecture.

\subsubsection{Ideals of nilpotent orbit closures and spherical Schubert varieties}
\label{nilorbits}
As discussed in \cite[\S 5]{KLMW}, one possible application of Conjecture \ref{conj:KLMW} and a related standard monomial theory is to describing the ideals of nilpotent orbit closures over fields of arbitrary characteristic.  More precisely, for any partition $\lambda \vdash n$ there is a corresponding nilpotent orbit $\mathbb{O}^\lambda \subset \fgl_n$, and it is an interesting question to describe generators of the ideal defining $\overline{\mathbb{O}^\lambda}$. By work of Lusztig \cite{Lus}, the nilpotent orbit closures are isomorphic to intersections of spherical Schubert varieties with the big cell of $\Gr_{\SL_n}$. Therefore ideal generators for spherical Schubert varieties will give ideal generators for nilpotent orbit closures.

In characteristic zero, Weyman \cite{W} has given explicit ideal generators of type-A nilpotent orbit closures.  An independent proof in the characteristic zero case is also given in \cite[Theorem 5.4.3]{KLMW}, using the isomorphism between nilpotent orbit closures $\overline{\mathbb{O}^\lambda}$ and intersections of spherical Schubert varieties and the big cell. However, it is not known whether Weyman's generators are correct in positive characteristic.

A broad generalization of Weyman's result is a conjecture due to Finkelberg and Mirkovi\'c \cite{FM} and expanded upon in \cite{KWWY}
about the defining equations for spherical Schubert varieties in the affine Grassmannian of a semi-simple group. In characteristic zero, with Kamnitzer, \cite{KMWY},\cite{KMW}, we have proved this conjecture in the type A case; however, this proof relies on Weyman's results on nilpotent orbits. One longer term goal of our line of investigation is to prove the type A conjecture in all characteristics and without bootstrapping from Weyman's work. In particular, this will generalize Weyman's theorem to all characteristics. Our hope is that one can attack this problem by better understanding the coordinate ring of spherical Schubert varieties via the approach of KLMW and the present paper.

\subsubsection{Studying $\cS^T$ and $\cG^T$}
\label{sec:studyingTinvs}
Our discussion suggests that (\ref{eq:65})
is an interesting finite-dimensional analogue of \eqref{eq:62}. 
More precisely, we have the following analogy between the infinite and finite dimensional settings:
\begin{center}
\begin{tabular}{ c|c} 
infinite dim & finite dim  \\
\hline
 $\cF$ & $\bigwedge^kV$ \\ 
 $\SGr$ & $\Gr(k,V)$ \\
 $\Gr_{\GL_n}$ & $\cG^T$ \\
  $\Gr_{\SL_n}$ & $\cS^T$ \\
\end{tabular}
\end{center}
Motivated by this, we formulate a KLMW Conjecture in the finite dimensional setting: we conjecture the following diagram to be Cartesian, in analogy with \eqref{eq:intro-klmw-diagram}:
\begin{equation}
\label{diagram:cart2}
\begin{tikzcd}
 (\cG^T)_{red} \arrow[d,hookrightarrow]                    \arrow[r,hookrightarrow] &  \VV(sh_\bullet^T)  \arrow[d,hookrightarrow] \\
 \Gr(k,V)                  \arrow[r,hookrightarrow] &  \PP(\bigwedge^kV) \
\end{tikzcd}
\end{equation}
Here $\VV(sh_\bullet^T)$ denotes the vanishing of the $T$-shuffle operators, which are used to define $\cS^T$, and $(\cG^T)_{red}$ is the reduced induced scheme structure on $\cG^T$.  (The scheme $\cG^T$ is in general not reduced.)  Equivalently, we conjecture (Conjecture \ref{conj:finiteKLMW}) that $\cS^T=(\cG^T)_{red}$.  Note that the $\FF$-points of $\cS^T$ and $\cG^T$ agree for any field $\FF$ (Proposition \ref{prop:Fpoints}).
We  mention that $\cG^T$ has been studied by Shayman  from a different perspective \cite{S}.

Finally, we conjecture that the ideal of $T$-shuffle equations defining $\cS^T$ is saturated (Conjecture \ref{conj:saturation}). This would provide an explicit description of the homogeneous coordinate ring of $\cS^T$. Computer experiments confirm this for small examples and for small characteristics of the ground field. 

\subsubsection{The KLMW basis}
The work of \cite{KLMW} provides a basis (indexed by $n$-regular partitions) for the degree-one part of the homogeneous coordinate ring. In the appendix, we prove some preliminary results about this basis: we show that the straightening algorithm in \cite{KLMW} is compatible with the dominance order on partitions (Proposition \ref{prop:lower-in-dominance}), and relate the KLMW basis to specializations of the Kostka-Foulkes matrix (Proposition \ref{prop:KF}). We expect that to develop an explicit standard monomial theory for $\Gr_{\SL_n}$ in the style of Hodge, we would need to further investigate this basis.

\subsection{Acknowledgements}
We thank Peter Samuelson for many conversations about Fock space and symmetric functions, and we thank Anthony Henderson for helpful remarks on a preliminary draft of this paper.
D.M. is supported by a PIMS Postdoctoral Fellowship. A.W. was supported in part by a Government of Ontario graduate scholarship.  O.Y. is supported by the Australian Research Council.   
This research is supported in part by Perimeter Institute for Theoretical Physics. Research at Perimeter Institute is supported by the Government of Canada through Industry Canada and by the Province of Ontario through the Ministry of Economic Development and Innovation.

\section{Preliminaries}

\subsection{Vector spaces}
Let $\Bbbk$ be an arbitrary commutative ring. We will work with schemes and ind-schemes over $\Bbbk$. We will abuse terminology by calling  free modules over $\Bbbk$ ``vector spaces'' and write ``dimension'' for their rank over $\Bbbk$. Given a finite-dimensional vector space $V$, we will identify $V$ with the scheme $\Spec \Sym V^*$. Similarly, we will abuse notation and write $\Bbbk$ for the affine line $\AA^1_\Bbbk$ thought of as a commutative ring object in schemes.

\subsection{Vector-valued polynomials}
\label{sec:vectorspaces}

Let $V$ and $W$ be vector spaces. Consider an algebraic map $\phi : V \rightarrow W$.
We say that $\phi$ is a {\it degree-$d$ vector-valued polynomial} if for all linear functionals $\lambda : W \rightarrow \Bbbk$, the composed map $ \lambda \circ \phi: V \rightarrow \Bbbk$
is a degree-$d$ polynomial on $V$.

Given a degree-$d$ vector valued polynomial $\phi : V \rightarrow W$, we define the {\it vanishing locus} $\VV(\phi)$ as a subscheme of the projective space $\PP(V)$ by
\begin{align}
 \VV(\phi) = \bigcap_{\lambda \in W^*} \VV( \lambda \circ \phi) 
\end{align}
where $\VV( \lambda \circ \phi)$ is the usual subscheme of $\PP(V)$ given by the vanishing of the homogeneous polynomial $\lambda \circ \phi$.

\subsection{Finite-dimensional Grassmannians}
\newcommand{\Fl}{\cF l}
Let us consider a scheme $S$ with a finite-rank vector bundle $\cF$. We say that a subsheaf $\cG \subseteq \cF$ is a {\it subbundle} if $\cG$ is itself a vector bundle, and the quotient $\cF/\cG$ is a vector bundle. 

Let us fix a $n$-dimensional vector space $V$. Then we will write $\Gr(k,V)$ for the {\it Grassmannian} of $k$-planes
in $V$. Explicitly, the functor of points of $\Gr(k,V)$ is given as follows. Given a test scheme $S$, let us write $\underline{V}$ for the trivial vector bundle on $S$ with fiber $V$. The set of $S$-points of $\Gr(k,V)$ is equal to the set of rank-$k$ vector subbundles of $\underline{V}$ (see e.g. \cite[\S 8.4]{GW} for more details on this modular definition).
In the next section, we will recall the classical perspective where $\Gr(k,V)$ is given by the vanishing of explicit Pl\"ucker equations inside of a projective space.

Similarly, for $k,\ell$  integers with $n \geq k \geq \ell \geq 0$, let $ \Fl_{k,\ell}(V)$ denote the incidence variety:
\begin{align}
  \label{eq:44}
  \Fl_{k,\ell}(V) = \Big\{ (U,W) \in \Gr(k,V) \times \Gr(\ell,V) \suchthat U \supseteq W \Big\}
\end{align}
Here we view \eqref{eq:44} as the definition of a moduli functor, that is, we view $U$ and $W$ as subbundles of the trivial bundle with fiber $V$ over a test scheme $S$.

\subsection{Notations for exterior algebras}
\label{section:clifford}

Let us fix a $n$-dimensional vector space $V$, and an integer $k$ with $0 \leq k \leq n$. Let us consider the exterior power $\bigwedge^{k} V$, and set $[k] = \{1,\cdots, k\}$. In particular $[0] = \emptyset$.

Suppose we have fixed vectors $v_1, \cdots, v_k \in V$. Then we will use the following notational short hand:
\begin{align}
  v_{[k]} = v_1 \wedge \cdots \wedge v_k 
\end{align}
Similarly, given any subset $R \subset [k]$, where $R = \{r_1, \cdots, r_d \}$ and $r_1 < \cdots < r_d$, we write:
\begin{align}
  v_{R} = v_{r_1} \wedge \cdots \wedge v_{r_d} \in \bigwedge^d V
\end{align}
Suppose $T : V \rightarrow V$ is a linear operator. Then for $v_{R}$ as above, we write:
\begin{align}
  \label{eq:57}
 Tv_{R} =   Tv_{r_1} \wedge \cdots \wedge Tv_{r_d}
\end{align}
The following sign appears often, so we introduce the following notation:
\begin{align}
  \label{eq:sign-R-minus-one}
 (-1)^{R-1} = (-1)^{r_1 - 1} \cdots  (-1)^{r_d - 1}
\end{align}

For $\alpha \in V^*$ we let $\iota_\alpha:\bigwedge^kV \to \bigwedge^{k-1}V$ denote the interior product with respect to $\alpha$.  Given $\eta=\alpha_1\wedge\cdots\wedge\alpha_\ell \in \bigwedge^\ell V^*$ let $\iota_\eta=\iota_{\alpha_1}\cdots\iota_{\alpha_\ell}$.  
Fixing once and for all dual bases $\{e_i\}$, $\{e_i^\ast\}$ for $V$ and $V^\ast$, respectively, we let 
$$\psi_i:\bigwedge^kV \to \bigwedge^{k+1}V, \quad \psi_i^*:\bigwedge^kV \to \bigwedge^{k-1}V$$
be the operators of multiplication by $e_i$, and 
taking the interior product with respect to $e_i^*$.
These satisfy the \textit{Clifford relations}:
\begin{align}
    \psi_i \psi_j + \psi_j \psi_i &= 0, \quad \psi_i^2 = 0, \nonumber \\
  \psi^*_i \psi^*_j + \psi^*_j \psi^*_i &= 0, \quad (\psi_i^*)^2 = 0, \label{eq: clifford relations} \\
  \psi_i \psi^{*}_j + \psi^{*}_j \psi_i &= \delta_{i,j} \nonumber
\end{align}
For any finite set $I$, where $I = \{ i_1, \cdots, i_r\}$ and $i_1 < i_2 \cdots < i_r$, we set:
\begin{align}
  \psi_{I} = \psi_{i_1} \cdots \psi_{i_r}, \quad  \psi^*_{I} = \psi^*_{i_1} \cdots \psi^*_{i_r}
\end{align}

\begin{Remark}
The above relations correspond to the Clifford algebra on the vector space spanned by the $\psi_i, \psi_i^\ast$, with the quadratic form
$ Q\Big( \sum_i(a_i \psi_i + b_i \psi_i^\ast) \Big) = \sum_i a_i b_i $.
\end{Remark}


\section{The KLMW conjecture and shuffle equations}
\subsection{Our main objects of study}

\subsubsection{The Sato Grassmannian}
\begin{Definition}
The {\bf Sato Grassmannian} is the covariant functor $\SGr^\bullet: \Bbbk-\operatorname{Algebras} \longrightarrow \operatorname{Sets}$, defined by:
\begin{align*}
\SGr^\bullet (R) = \left\{ \Lambda \subset R((t)) : \begin{array}{l} \text{(a) } \Lambda \text{ is an } R\text{--submodule}, \\ \text{(b) } R((t)) / \Lambda \text{ is projective over } R, \\ \text{(c) } t^N R[[t]] \subset \Lambda \subset t^{-N} R[[t]] \text{ for some } N\geq 0 \end{array} \right\} 
\end{align*}
\end{Definition}

Notice that when $R = \FF $ is a field we recover the ``usual'' Sato Grassmannian, as defined e.g. in \cite[Chapter 7]{PS} or \cite[Section 2]{KLMW} (where they refer to this space as the ``infinite Grassmannian'').  In this case the $\FF$--points consist of subspaces 
$$ \Lambda \subset \FF((t)) $$
such that $t^N \FF[[t]] \subset \Lambda \subset t^{-N} \FF[[t]]$ for some $N\geq 0$.

\begin{Remark}
Give $R$ the discrete topology, and consider $R((t))$ as a topological $R$--module with the submodules $t^N R[[t]]$ forming a base of neighbourhoods of $0$.  (In particular, an $R$--submodule $\Lambda \subset R((t))$ is open iff it contains some $t^N R[[t]]$.)  Call a submodule $\Lambda$ a lattice if it is open and $\Lambda / U$ is finitely generated for all open submodules $U\subset \Lambda$; this is equivalent to requiring that 
$$t^{N} R[[t]] \subset \Lambda \subset t^{-N} R[[t]]$$
for some $N$. Hence $\SGr^\bullet(R)$ is precisely the set of coprojective lattices in $R((t))$, in agreement with the Co-Sato Grassmannian as defined in \cite[Section 3.7]{D} (see also \cite[Section 1.5]{Z}).
\end{Remark}

\newcommand{\SGrtrunc}[1]{\SGr_{\leq #1}}

$\SGr^\bullet$ is filtered by the subfunctors $\SGrtrunc{N}^\bullet$ defined by:
\begin{align}
  \label{eq:64}
 \SGrtrunc{N}^\bullet(R) = \left\{ \Lambda \in \SGr^\bullet(R): t^N R[[t]] \subset \Lambda \subset t^{-N} R[[t]] \right\}
\end{align}
We define the \textit{virtual dimension} of $\Lambda \in  \SGr^\bullet$ as:
\begin{align}
  \label{eq:69}
\vdim(\Lambda)=\dim(R[[t]]/(\Lambda\cap R[[t]]))-\dim(\Lambda/(\Lambda\cap R[[t]]))
\end{align}
(Note that this is the negative of the virtual dimension defined in \cite{KLMW}.)
Let $\SGr^{(c)}(R) = \left\{ \Lambda \in \SGr^\bullet(R): \vdim(\Lambda)=c \right\}$, 
and similarly define $\SGrtrunc{N}^{(c)}$.  Then
\begin{equation}
\label{eq:SGrchargec}
\SGrtrunc{N}^{(c)} \cong \Gr(N-c,t^{-N}R[[t]]/t^NR[[t]])
\end{equation}
 via the map $\Lambda \mapsto \Lambda/t^NR[[t]]$.  
Indeed, 
\begin{align*}
\dim (\Lambda/t^NR[[t]]) &= \dim(\Lambda/(\Lambda\cap R[[t]]))+\dim(\Lambda\cap R[[t]]/t^NR[[t]]) \\
&= \dim(\Lambda/(\Lambda\cap R[[t]]))+\dim(R[[t]]/t^NR[[t]])-\dim(R[[t]]/(\Lambda\cap R[[t]])) \\
&= N-c
\end{align*} 
Let $V_{[-N,N)} = t^{-N} R[[t]] /  t^{N} R[[t]]$, viewed as a $\Bbbk$-scheme. Then we have the following:

\begin{Proposition}
\label{Prop: A finite Gr appears}
The functor $\SGrtrunc{N}^\bullet$ is represented by a projective scheme, isomorphic to the disjoint union of finite-dimensional Grassmannians:
$$ \Gr(V_{[-N,N)}) = \bigsqcup_{0\leq k \leq 2N} \Gr(k,V_{[-N,N)}) $$
\end{Proposition}

\subsubsection{The Fock space}

Let us consider $R((t))$ as a module over $R$, and let us write $e_i = t^i \in R((t))$ for $i \in \ZZ$.
Given a sequence of integers $i_k$ for $k \geq 0$ such that $i_1 < i_2 < \cdots$, and $i_{k+1} = i_k + 1$ for $k$ sufficiently large, we can form the ordered semiinfinite wedge product:
\begin{align}
  \label{eq:3}
 e_{i_1} \wedge e_{i_2} \wedge \cdots 
\end{align}
We form the {\it Fermion Fock space}, which is the formal span of ordered semi-infinite wedges:
\begin{align}
  \label{eq:1}
  \cF^\bullet = \Span_R \{ e_{i_1} \wedge e_{i_2} \wedge \cdots \suchthat i_{k+1} = i_{k} + 1 \text{ for } k>>0 \}
\end{align}
One can also form semi-infinite wedge products where we drop the ordered condition $i_1 < i_2 < \cdots$, and consider sequences instead which are {\it eventually increasing} (i.e. can be made increasing by an element of the ind-finite group $S_\infty$).  We declare such wedge products to be equal to ordered semi-infinite wedge products up to a sign given by the usual rule.

Let $c \in \ZZ$. Then we define the charge-$c$ Fock space by:
\begin{align}
  \label{eq:4}
 \cF^{(c)} =  \Span_R \{ e_{i_1} \wedge e_{i_2} \wedge \cdots \suchthat i_{k} = (k-1) + c \text{ for } k>>0 \}
\end{align}
Then we have:
\begin{align}
  \label{eq:5}
  \cF^\bullet = \bigoplus_{c \in \ZZ} \cF^{(c)}
\end{align}
We will focus on the charge-$0$ Fock space, so to simplify notation we will write $\cF = \cF^{(0)}$. 

\subsubsection{The embedding  $\SGr^{(c)} \hookrightarrow \PP \left(\cF^{(c)}\right)$}
For $\ell \leq 2N$ we have an inclusion of $R$-modules
\begin{align}
  \label{eq:9}
  \bigwedge^{\ell} ( t^{-N} R[[t]]/t^{N}R[[t]] ) \hookrightarrow \cF^{(N-\ell)}
\end{align}
given by:
$v_{1} \wedge \cdots \wedge v_{\ell} \mapsto v_{1} \wedge \cdots \wedge v_{\ell} \wedge e_{N} \wedge e_{N+1} \wedge \cdots$.

Fix $c\in \ZZ$.  For any $N,\ell$ such that $N-\ell=c$  we have by (\ref{eq:SGrchargec}) and the Pl\"ucker embedding:
\begin{align}
  \label{eq:10}
 \SGrtrunc{N}^{(c)} = \Gr(\ell,  t^{-N} R[[t]]/t^{N}R[[t]] )  \hookrightarrow \PP\left( \bigwedge^{\ell} ( t^{-N} R[[t]]/t^{N}R[[t]] \right)
\end{align}
These embeddings are compatible with the embeddings $\SGr_{\leq N} \hookrightarrow \SGr_{\leq N'}$ and the embeddings:
\begin{align}
  \label{eq:2}
  \bigwedge^{\ell} ( t^{-N} R[[t]]/t^{N}R[[t]] ) \hookrightarrow \bigwedge^{\ell+(N'-N)} ( t^{-N'} R[[t]]/t^{N'}R[[t]] )
\end{align}
Thus in the limit, we have a projective embedding:
\begin{align}
  \label{eq:11}
  \SGr^{(c)} \hookrightarrow \PP \left(\cF^{(c)}\right)
\end{align}
Here the right-hand side is viewed as an ind-projective ind-scheme, corresponding to the limit of the inclusions (\ref{eq:2}).

\begin{Remark}
As with $\cF$, we will be interested only in $\SGr^{(0)}$, so to simplify notation we write $\SGr=\SGr^{(0)}$.
\end{Remark}

\subsubsection{The affine Grassmannians of $\GL_n$ and $\SL_n$}
\label{sec:affgrass}

\begin{Definition} 
The {\bf affine Grassmannian of ${\GL_n}$} is the covariant functor $\Gr^\bullet_{\GL(n)}: \Bbbk-\operatorname{Algebras} \longrightarrow \operatorname{Sets}$ defined by:
$$ \Gr^\bullet_{\GL_n}(R) = \left\{ \Lambda \subset R((t))^n : \begin{array}{l} \text{(a) } \Lambda \text{ is a finitely generated } R[[t]]\text{--submodule}, \\ \text{(b) } \Lambda \text{ is projective over } R[[t]], \\ \text{(c) } \Lambda \otimes_{R[[t]]} R((t)) = R((t))^n \end{array} \right\} $$
\end{Definition}
\noindent The decomposition of $\Gr^\bullet_{\GL_n}$ into connected components is given by 
\begin{align}
\Gr^\bullet_{\GL_n}=\bigsqcup_{c\in \ZZ} \Gr_{\GL_n}^{(c)},
\end{align}
where $\Gr_{\GL_n}^{(c)}=\Gr^\bullet_{\GL_n}\cap \SGr^{(c)}$. Again, to simplify notation we will write $\Gr_{\GL_n}= \Gr_{\GL_n}^{(0)}$.
The following theorem is well-known, see e.g. \cite[Theorem 1.1.3]{Z} or \cite[Section 2]{G}.

\begin{Theorem}
$\Gr_{\GL(n)}$ is an ind-projective ind-scheme.
\end{Theorem}

Consider the sub-ind-scheme $\SGr^{t^n}$ of $\SGr$, defined by:
\begin{align}
  \label{eq:70}
 \SGr^{t^n}(R) = \Big\{ \Lambda \in \SGr(R) : t^n \Lambda \subset \Lambda \Big\} 
\end{align}
Fix an $R((t))$-basis $\{e_1,\ldots, e_n\}$ for $R((t))^n$.  We identify $R((t))^n \stackrel{\sim}{\longrightarrow} R((t))$, via:
\begin{align}
  \label{eq:71}
t^k e_r \mapsto t^{kn+r} , \qquad k\in \ZZ, 1\leq r\leq n
\end{align}

From the proof of \cite[Theorem 1.1.3]{Z} we deduce the following.

\begin{Proposition}
\label{prop:Gr_GLn}
The isomorphism $R((t))^n \stackrel{\sim}{\rightarrow} R((t))$ identifies $\Gr_{\GL_n}$ with $\SGr^{t^n}$.  In particular, we can consider $\Gr_{\GL_n}$ as a closed sub-ind-scheme of $\SGr$.
\end{Proposition}

\begin{Definition}
  The {\bf affine Grassmannian $\Gr_{\SL_n}$} of $\SL_n$ is the closed subfunctor $\Gr_{\SL_n} \subseteq \Gr_{\GL_n}$ defined as follows:
  \begin{align}
    \Gr_{\SL_n}(R) = \left\{ \cL \in  \Gr_{GL_n}(R) :  \bigwedge^n_{R[[t]]} \cL = R[[t]] \right\} \nonumber
  \end{align}
\end{Definition}
Note that $\Gr_{\GL_n}$ is a non-reduced ind-scheme, and $\Gr_{\SL_n}=(\Gr_{\GL_n})_{red}$ by \cite[Proposition 6.4]{BL}.

\subsubsection{Big cells}
\label{sec:initial-def-of-big-cells}
Let $G$ denote either $GL_n$ or $SL_n$. Then we can form the group ind-scheme $G_1[t^{-1}]$ whose $R$-points are given by:
  \begin{align}
    \label{eq:6}
  G_1[t^{-1}](R) = \{ g \in G(R[t^{-1}]) \suchthat g(\infty) = 1 \}
  \end{align}
When $G = GL_n$, the group $G_1[t^{-1}](R)$ consists of matrix-valued polynomials in $t^{-1}$ whose determinant is a unit in $R[t^{-1}]$, and whose constant term is the identity matrix.  Similarly, when 
$G = SL_n$, the group $G_1[t^{-1}](R)$ consists of matrix-valued polynomials whose determinant
is equal to $1$, with constant term the identity. 
  
The group $G_1[t^{-1}]$ acts on $\Gr_G$. Observe that we have distinguished point $1 \in \Gr_G$. By acting on this distinguished point, we produce a map $G_1[t^{-1}] \rightarrow \Gr_G$ that is known to be an open embedding (see e.g. \cite[Theorem 2.5]{BL}). Let us write $\Gr_G^\circ$ for the corresponding open sub-ind-scheme, which we call the {\bf big cell} of $\Gr_G$.

\subsubsection{The main theorem}
Consider the set $\mathfrak{S}_n$ of linear forms on $\PP(\cF)$ that vanish on $\Gr_{\SL_n}$.  Let $V(\Lambda_0) \subseteq \cF$ be the subspace of $\cF$ given by the vanishing of these linear forms. Over a field of characteristic zero, $V(\Lambda_0)$ is identified with the basic representation of affine $\SL_n$ (hence the notation). Therefore we have natural embeddings 
\begin{align}
\label{eq:12.5}
    \Gr_{\SL_n} \hookrightarrow \PP(V(\Lambda_0)) \hookrightarrow \PP(\cF).
\end{align}
Combining (\ref{eq:12.5}) with the embeddings from (\ref{eq:11}) and Proposition \ref{prop:Gr_GLn} we get a commutative square:
\begin{equation}
  \label{eq:klmw-diagram}
\begin{tikzcd}
 \Gr_{\SL_n} \arrow[d,hookrightarrow]                    \arrow[r,hookrightarrow] &  \PP(V(\Lambda_0))  \arrow[d,hookrightarrow] \\
 \SGr                  \arrow[r,hookrightarrow] &  \PP(\cF) \
\end{tikzcd}
\end{equation}
The following is the main theorem of this paper, which is a proof of the KLMW Conjecture (Conjecture \ref{conj:KLMW}).

\begin{Theorem}
  \label{thm:klmw-conjecture}
 We have the following equality of ind-schemes in $\PP(\cF)$: 
  \begin{align}
\Gr_{SL_n} =     \PP\left(V(\Lambda_0)\right) \cap \SGr  
  \end{align}
 \end{Theorem}

The proof will be completed in Section \ref{section:proofKLMW}.

\subsection{A theorem of Kreiman, Lakshmibai, Magyar, and Weyman}

In \cite{KLMW} the authors explicitly describe $\mathfrak{S}_n$ 
in terms of certain explicit ``shuffle operators''. We will now briefly recall this result, but reformulated in the language of Clifford operators.   

\subsubsection{The Clifford action}
\label{section:cliffordaction}

The Fermion Fock space $\cF^\bullet$ naturally carries an action of a Clifford algebra. The Clifford algebra is generated by operators $\psi_{i}$ and $\psi^*_i$ for $i \in \ZZ$ subject to the Clifford relations (Section \ref{section:clifford}). 
The Clifford algebra acts on $\cF^\bullet$ by the following formulas. For $v \in \cF^\bullet$ 
\begin{align}
  \label{eq:16}
  \psi_i(v) = e_i \wedge v
\end{align}
and
\begin{align}
  \label{eq:17}
  \psi^{*}_i( v) =
  \begin{cases}
    w \text{ if } v = e_i \wedge w \text{ for some } w \in \cF^\bullet \\
    0 \text{ otherwise }
  \end{cases}
\end{align}
The operator $\psi_i$ (resp. $\psi^*_i$) has degree $1$ (resp. degree $-1$) with respect to the charge grading on $\cF^\bullet$.

\begin{Remark}
The action of the Clifford algebra on $\cF^\bullet$ is faithful \mbox{\cite[Lemma 3.3]{T1}}.  Therefore if an operator on $\cF^\bullet$ can be expressed in terms of elements of the Clifford algebra, then this expression is necessarily unique.  This applies in particular to the operators $\mathrm{sh}^{(n)}_d$ considered below.
\end{Remark}

The dual space $\cF^\ast$ is the formal {\it completed} span of ordered semi-infinite dual wedges:
\begin{equation}
\cF^\ast = \overline{\Span}_\Bbbk \Big\{ e_{i_1}^\ast \wedge e_{i_2}^\ast \wedge \cdots \suchthat i_k = k-1 \text{ for } k >>0\Big\},
\end{equation}
where $\overline{\Span}_\Bbbk$ here simply means the direct product of 1-dimensional spaces $ \Bbbk \cdot e_{i_1}^\ast \wedge e_{i_2}^\ast\wedge \cdots$.  The pairing between $\cF^\ast$ and $\cF$ is defined as follows: if $i_1 < i_2 < \ldots$ and $j_1 < j_2 < \ldots$ are increasing sequences, then 
\begin{equation}
\langle e_{i_1}^\ast \wedge e_{i_2}^\ast\wedge\cdots, \ e_{j_1}\wedge e_{j_2}\wedge \cdots \rangle = \prod_{t\geq 1} \delta_{i_t, j_t}
\end{equation}
This is well-defined since the sequences differ in only finitely many components.  The full dual Fock space $\cF^{\bullet,\ast}$ is defined similarly (we define its charge-$c$ part analogously to (\ref{eq:4})).

The Clifford algebra acts on $\cF^{\bullet,\ast}$ similarly to its action on $\cF^\bullet$:
\begin{equation}
\psi_i(v) = e_i^\ast \wedge v
\end{equation}
and 
\begin{equation}
\psi_i^\ast(v) = 
  \begin{cases}
    w \text{ if } v = e_i^\ast \wedge w \text{ for some } w \in \cF^{\bullet, \ast} \\
    0 \text{ otherwise }
  \end{cases}
\end{equation}
Note that $\psi_i$ is the adjoint of $\psi_i^\ast$.

\subsubsection{Shuffle operators}

Let $n \geq 2$ be an integer. Then for each $d \geq 1$, we define the shuffle operator:
\begin{align}
  \label{eq:18}
  \mathrm{sh}^{(n)}_d = \sum_{j_1<\cdots<j_d}\psi_{j_d+n}\cdots\psi_{j_1 +n} \psi_{j_1}^\ast\cdots \psi_{j_d}^\ast
\end{align}
We consider $\mathrm{sh}^{(n)}_d$ as an operator $ \mathrm{sh}^{(n)}_d : {\cF} \rightarrow {\cF}$.  Note that the adjoint operator $(\mathrm{sh}^{(n)}_d)^\ast: \cF^\ast \rightarrow \cF^\ast$ is given by
\begin{align}
  (\mathrm{sh}^{(n)}_d)^\ast = \sum_{j_1<\cdots<j_d}\psi_{j_d-n}\cdots\psi_{j_1 -n} \psi_{j_1}^\ast\cdots \psi_{j_d}^\ast
\end{align}

\begin{Theorem}{\cite{KLMW}}
\label{thm:KLMW}
  The set of all linear forms on $\cF$ vanishing on $V(\Lambda_0)$ is given by the images of the adjoint shuffle operators:
 \begin{align*}
 \mathfrak{S}_n= \sum_{d \geq 1}  \mathrm{im}(\mathrm{sh}^{(n)}_d)^\ast 
 \end{align*}
\end{Theorem}
We consider the linear operators $sh^{(n)}_d$ as degree-$1$ vector-valued polynomials on $\cF$. Therefore, we can consider their vanishing locus $\VV(sh^{(n)}_d)$, which is a sub ind-scheme of $\PP(\cF)$. Let us write:
\begin{align}
  \label{eq:46}
\VV\left(sh^{(n)}_\bullet\right) =  \bigcap_{d} \VV\left(sh^{(n)}_d\right)
\end{align}
Note that $\VV(sh^{(n)}_\bullet)=\PP(V(\Lambda_0))$ by Theorem \ref{thm:KLMW}.


\section{Symmetric function interpretation of shuffles and Fock space}
\label{sec:sym-functions}
\newcommand{\Maya}{\mathtt{Maya}}

\newcommand{\rightarrowisom}{\overset{\sim}{\longrightarrow}}

\subsection{Fock space and symmetric functions}
\label{section:Fock}
There is a well-known linear isomorphism
\begin{equation}
\Sym  \rightarrowisom \cF
\end{equation}
where $\Sym$ is the ring of symmetric functions (see \cite{M} for a thorough treatment of $\Sym$).  For an increasing sequence $i_1< i_2 <\ldots$ of charge 0 as in (\ref{eq:3}) (i.e. $i_k = k-1$ for $k>>0$), the isomorphism sends
\begin{equation}
s_\lambda \mapsto e_{i_1}\wedge e_{i_2}\wedge\cdots,
\end{equation}
where the partition $\lambda = \big(i_1, i_2 - 1,\ldots, i_k - (k-1),\ldots \big)$. See  \cite[\S 14.9--14.10]{K} or \cite{T2} for more details on this isomorphism (cf. Appendix \ref{appendix:2}).

Similarly, there is an isomorphism
\begin{align}
  \label{eq:33}
 \overline{\Sym} \rightarrowisom \cF^*,
\end{align}
where by $\overline{\Sym}$, we mean the completion of the vector space of symmetric functions with respect to its usual grading.  We think of $\overline{\Sym} \cong \Sym^\ast$, via the pairing for which the Schur functions are orthonormal.  In particular, with the notation above this isomorphism sends
\begin{equation}
s_\lambda \mapsto e_{i_1}^\ast \wedge e_{i_2}^\ast \wedge \cdots 
\end{equation}

\subsection{Shuffle operators and Frobenius twists}
\label{section: Shuffle operators and Frobenius twists}

For each $d\geq 1$, consider the operator $\alpha_d : \cF^\ast \rightarrow \cF^\ast$:
\begin{align}
\label{eq: power sum Clifford}
\alpha_d = \sum_{j} \psi_{j-d} \psi_j^\ast
\end{align}
Because of the identification $\overline{\Sym} \rightarrowisom  \cF^*$ from (\ref{eq:33}), we can transfer the action of the operators $\alpha_d$ and $(sh_d^{(n)})^*$ to $\overline{\Sym}$.  In fact, we will see that these operators are nothing but multiplication by certain elements of $\Sym$ (see Lemma \ref{lemma: power sum as Clifford} and Corollary \ref{cor: end as Clifford} below).

 The following result can be thought of as part of the Boson-Fermion correspondence \cite[\S 14.10]{K}, or as a restatement of the Murnaghan-Nakayama rule \cite[\S I.3, Example 11]{M}:

\begin{Lemma}
\label{lemma: power sum as Clifford}
Under $\overline{\Sym} \rightarrowisom \cF^\ast$, multiplication by $p_d$ on $\Sym$ is identified with the operator $\alpha_d$ on $\cF^\ast$.
\end{Lemma}

More generally, for any $f\in \Sym$ we can ask how to write the operator of multiplication by $f$ on $\overline{\Sym}$ in terms of Clifford algebra elements acting on $\cF^\ast$.  We will now give an answer to this question.  

We will think of an element $f\in \Sym$ as an infinite sum of monomials in a countable set of variables $\{x_i\}_{i\in \ZZ}$, where any monomial contains only finitely many variables.  Take a partition $\lambda = (\lambda_1,\ldots, \lambda_\ell)$ where $\lambda_1\geq \cdots \geq \lambda_\ell >0$; we also write $\lambda = (1^{k_1},2^{k_2},\ldots)$ in exponential notation.  Consider the associated monomial symmetric function $m_\lambda \in \Sym$:
\begin{equation}
\label{eq: monomial symmetric function}
m_\lambda = \sum_{i_1<\ldots < i_\ell} \sum_{\alpha \in S_\ell \lambda} x_{i_1}^{\alpha_1}\cdots x_{i_\ell}^{\alpha_\ell} =
\frac{1}{ \prod_{t\geq 1} k_t !} \sum_{\substack{i_1,\ldots, i_\ell \\ \text{distinct} }} x_{i_1}^{\lambda_1} \cdots x_{i_\ell}^{\lambda_\ell}
\end{equation}
Here $S_\ell \lambda$ denotes the set of permutations of the set $\{\lambda_1,...,\lambda_\ell\}$.
Define a map of $\ZZ$--modules $\mmap: \Sym \rightarrow \End(\cF^\ast)$, by 
\begin{align}
\mmap( m_\lambda) &:=\sum_{i_1<\ldots < i_\ell} \left( \sum_{\alpha \in S_\ell \lambda} \psi_{i_\ell - \alpha_\ell} \cdots \psi_{i_1 - \alpha_1} \right) \psi_{i_1}^\ast \cdots \psi_{i_\ell}^\ast \label{eq: Sym as Clifford}\\
& = \frac{1}{ \prod_{t\geq 1} k_t !} \sum_{\substack{i_1,\ldots, i_\ell \\ \text{distinct} }} \psi_{i_\ell - \lambda_\ell}\cdots \psi_{i_1-\lambda_1} \psi_{i_1}^\ast \cdots \psi_{i_\ell}^\ast, \label{eq: Sym as Clifford 2}
\end{align}
extended by $\Bbbk$--linearity.  The fact that the expressions on the right-hand side are equal follows from the fact that we are dividing by $|\operatorname{Stab}_{S_\ell} \lambda| = \prod_{t\geq 1} k_t !$, using the Clifford commutation relations.  Note in particular that $\mmap(p_s) = \alpha_s$.

\begin{Proposition}
\label{prop: Sym as Clifford}
\mbox{}

\begin{enumerate}
\item[(a)] $\mmap$ is a ring homomorphism.

\item[(b)] Under $ \overline{\Sym} \rightarrowisom \cF^*$, multiplication by $f$ on $\overline{\Sym}$ is identified with the operator $\mmap(f)$ on $\cF^\ast$.
\end{enumerate}
\end{Proposition}

For the proof, we will make use of the following multiplication rule (making use of exponential notation):
\begin{equation}
\label{eq: prod ps and ms}
p_s m_{(1^{k_1},2^{k_2},\ldots)} = (k_s + 1) m_{(\ldots, s^{k_s + 1},\ldots)} + \sum_{t\geq 1} (k_{s+t}+1) m_{(\ldots,t^{k_t -1},\ldots, (s+t)^{k_{s+t}+1},\ldots)} 
\end{equation}
This is easily proven using the second description (\ref{eq: monomial symmetric function}) of $m_\lambda$, by multiplying by $p_s = \sum_j x_j^s$ and collecting monomials in the result.

\begin{proof}
It suffices to work over $\ZZ$.  First, we claim that $\mmap$ satisfies $\mmap(p_s) \mmap(m_\lambda) = \mmap(p_s m_\lambda)$ for all $r, \lambda$.  The proof is analogous to that of (\ref{eq: prod ps and ms}): since $\mmap(p_s) = \alpha_s$, by applying the Clifford algebra relations we have
\begin{align}
\mmap(p_s) \mmap(m_\lambda) & = \alpha_s \cdot\mmap(m_\lambda) \nonumber\\
  & = \frac{1}{\prod_{t\geq 1} k_t!} \sum_j \psi_{j-s} \psi_j^\ast \sum_{\substack{i_1,\ldots,i_\ell \\ \text{distinct}}} \psi_{i_\ell - \lambda_\ell}\cdots \psi_{i_1-\lambda_1} \psi_{i_1}^\ast \cdots \psi_{i_\ell}^\ast \nonumber \\
& = \frac{1}{\prod_{t\geq 1} k_t!} \sum_{\substack{j, i_1,\ldots,i_\ell \\ \text{distinct}}} \psi_{j - s} \psi_{i_\ell - \lambda_\ell}\cdots \psi_{i_1-\lambda_1} \psi_{i_1}^\ast \cdots \psi_{i_\ell}^\ast \psi_j^\ast \nonumber \\
&  + \frac{1}{\prod_{t\geq 1} k_t!} \sum_{t=1}^\ell \sum_{\substack{i_1,\ldots,i_\ell \\ \text{distinct}}} \psi_{i_\ell - \lambda_\ell}\cdots \psi_{i_t - \lambda_t - s} \cdots \psi_{i_1-\lambda_1} \psi_{i_1}^\ast \cdots \psi_{i_\ell}^\ast,  \label{eq: M prod ps and ms}
\end{align}
where the part labelled by $t$ in the above sum corresponds to $j = i_t - \lambda_t$.  Accounting for multiplicities, we see that the terms on the right-hand side of (\ref{eq: M prod ps and ms}) agree precisely with $\mmap$ applied to the right-hand side of (\ref{eq: prod ps and ms}).  This proves the claim.

Next, observe that $\Sym$ and $\End(\cF^\ast)$ are torsion free as $\ZZ$--modules.  Since the elements $p_s$ generate $\Sym_\QQ$, and $\{ m_\lambda\}$ is a $\ZZ$--basis for $\Sym$, the equation $\mmap(p_s m_\lambda) = \mmap(p_s)  \mmap(m_\lambda)$ holds for all $r,\lambda$ iff $\mmap$ is a homomorphism. This proves part (a).  Likewise, if two homomorphisms $\Sym \rightarrow \End(\cF^\ast)$ agree on the elements $p_s$  then they are equal.  We have two such homomorphisms: $\mmap$ on the one hand, and on the other the map sending $f\in \Sym$ to the operation of multiplication by $f$ under $\overline{\Sym} \cong \cF^\ast$.  Since $\mmap(p_s) = \alpha_s$, these homomorphisms agree by Lemma \ref{lemma: power sum as Clifford}.  This proves (b).
\end{proof}

Recall the notion of {\it Frobenius twist} on $\Sym$: for each positive integer $n$ there is a ring endomorphism $f\mapsto f^{(n)}$ of $\Sym$, where $f^{(n)}(...,x_i,x_{i+1},...)=f(...,x_i^n,x_{i+1}^n,...)$.  In particular, for $e_d = m_{(1^d)}$ we have $e_d^{(n)} = m_{(n^d)}$, and from this
we obtain the following corollaries.
\begin{Corollary}
\label{cor: end as Clifford}
Under $\overline{\Sym} \rightarrowisom \cF^\ast$, multiplication by $e_d^{(n)}$ on $\overline{\Sym}$ is identified with the operator $(sh_d^{(n)})^\ast$ on $\cF^\ast$.
\end{Corollary}

\begin{Corollary}
\label{cor:frobtwists}
The ideal $\langle f^{(n)} \rangle \subset \overline{\Sym}$ generated by all $n$--th Frobenius twists is contained inside $\VV(sh_\bullet^{(n)})$, under the isomorphism $\overline{\Sym} \rightarrowisom \cF^\ast$.
\end{Corollary}

\subsection{Big cells}

\subsubsection{The big cell of Grassmannians and Sato Grassmannian}
Recall that we construct the (charge-$0$) Sato Grassmannian $\SGr$ as the union of Grassmannians $\Gr(N,V_{[-N,N)})$ where:
\begin{align}
  V_{[-N,N)} = \Span \{ e_{-N}, e_{-N+1}, \cdots, e_{0}, \cdots, e_{N-1} \}
\end{align}
We define the big cell $\Gr(N,V_{[-N,N)})^\circ \subseteq \Gr(N,V_{[-N,N)})$ consisting of all subspaces where the Pl\"ucker coordinate $e_0^* \wedge \cdots \wedge e_{N-1}^*$ does not vanish. To elaborate, we have the Pl\"ucker
embedding $\Gr(N,V_{[-N,N)}) \hookrightarrow \PP(\bigwedge^{N}\left(V_{[-N,N)}\right))$. Let $D_N$ be the divisor given by the vanishing of $e_0^* \wedge \cdots \wedge e_{N-1}^*$. The Pl\"ucker embedding is the complete linear system associated to the line bundle $\cO(D_N)$. Therefore we obtain a map:
\begin{align}
  \label{eq:restriction-to-divisor-complement}
\bigwedge^{N}\left(V_{[-N,N)}\right)^*  \rightarrow  \cO(\Gr(N,V_{[-N,N)})^\circ)
\end{align}

We can also describe the map \eqref{eq:restriction-to-divisor-complement} group theoretically. Consider the map $\pi_N: \GL(V_{[-N,N)}) \rightarrow \Gr(N,V_{[-N,N)})$ given by:
\begin{align}
  g \mapsto g \cdot \Span\{e_{0},\cdots,e_{N-1}\},
\end{align}
Let $U(V_{[-N,N)})$  be the set of upper-triangular unipotent matrices (with respect to the basis $\{ e_{-N}, e_{-N+1}, \cdots, e_{0}, \cdots, e_{N-1} \}$). Then $\pi_N(U(V_{[-N,N)}) = \Gr(N,V_{[-N,N)})^\circ$.
For $v_0^* \wedge \cdots \wedge v_{N-1}^* \in \bigwedge^{N}(V_{[-N,N)})^*$,
we define a function
$
 v_0^* \wedge \cdots \wedge v_{N-1}^*  :  U^-(V_{[-N,N)}) \rightarrow \Bbbk
$
 by 
 $
  g \mapsto \langle v_0^* \wedge \cdots \wedge v_{N-1}^*, g \cdot e_0 \wedge \cdots e_{N-1} \rangle
  $.
This map descends to a map
$
 v_0^* \wedge \cdots \wedge v_{N-1}^*  : \Gr(N,V_{[-N,N)})^\circ   \rightarrow \Bbbk
 $.
Varying $v_0^* \wedge \cdots \wedge v_{N-1}^*$ we exactly obtain \eqref{eq:restriction-to-divisor-complement}.

We define the big cell $\SGr^\circ$, of the Sato Grassmannian to be the union:
\begin{align}
 \label{eq:36} 
  \SGr^\circ  = \bigcup_{N\geq1} \Gr(N,V_{[-N,N)})^\circ  
\end{align}
Taking the limit of \eqref{eq:restriction-to-divisor-complement} as $N$ tends to infinity, we obtain an injective linear map:
\begin{align}
  \label{eq:34}
  \cF^*  \hookrightarrow \cO ( \SGr^\circ) 
\end{align}
Identifying $\overline{\Sym}$ with $\cF^*$, we obtain an injective map:
\begin{align}
  \label{eq:4}
  \overline{\Sym}  \hookrightarrow \cO ( \SGr^\circ) 
\end{align}

Finally, we note that the intersections $\Gr^\circ_{GL_n} = \Gr_{GL_n} \cap \SGr^\circ$ and $\Gr^\circ_{SL_n} = \Gr_{SL_n}\cap \SGr^\circ$ recover the big cells from Section \ref{sec:initial-def-of-big-cells}.  Since these are open sub-ind-schemes, this is a set-theoretic statement, and follows for example by using the Birkhoff decomposition.

\subsubsection{The big cell $\Gr^\circ_{\GL_1}$}

We can identify the big cell of the $\GL_1$ affine Grassmannian $\Gr^\circ_{GL_1}$ with 
infinite upper-triangular Toeplitz matrices, with $R$--points
\begin{align}
  \label{eq:39}
  \begin{bmatrix}
    \ddots & \vdots & \vdots       & \vdots       & \vdots     & \vdots & \vdots \\
    \cdots &  1     &  h_1 & h_2 & h_3 & h_4 & \cdots \\
    \cdots &        &  1           &  h_1 & h_2 & h_3 & \cdots \\
    \cdots &        &              &  1      &  h_1 & h_2 & \cdots \\
    \cdots &        &              &         &  1 & h_1 & \cdots \\
    \cdots &        &              &         &    & 1 &  \cdots \\
    \vdots & \vdots & \vdots       & \vdots  & \vdots & \vdots & \ddots\\
  \end{bmatrix}
\end{align}
where all $h_i \in R$, subject to the following two conditions: 
\begin{enumerate}
\item $h_i = 0$ for $i \gg 0$
\item $1 + \sum_{i\geq1} h_i t^{-i}$ is invertible in $R[t]$
\end{enumerate}
These conditions show that $\Gr^\circ_{GL_1}$ is an ind-scheme with one closed point. We can restrict \eqref{eq:4} to obtain a map:
\begin{align}
  \label{eq:42}
  \overline{\Sym}  \hookrightarrow \cO ( \Gr^\circ_{GL_1}) 
\end{align}
As above, this map can be interpretted group theoretically. Given $v_0^* \wedge v_1^* \wedge \cdots \in \overline{Sym}$, the corresponding function in  $\cO ( \Gr^\circ_{GL_1})$ sends a matrix $X$ of the form \eqref{eq:39} to the minor $\langle v_0^* \wedge v_1^* \wedge \cdots, X \cdot e_0 \wedge e_1 \wedge \cdots \rangle$. Because $v_i = e_i$ for $i \gg 0$, this {\it a priori} infinite minor is equal to a finite minor.
 
Finally, we observe that the Jacobi-Trudi formula shows that \eqref{eq:42} remains injective. This motivates our notation $h_i$ above: they correspond to the homogeneous symmetric functions. 

\subsubsection{The big cell $\Gr^\circ_{\GL_n}$}
\newcommand{\Mat}{\mathrm{Mat}}
Similar to the $\GL_1$ case, we identify the big cell of the $\GL_n$ affine Grassmannian $\Gr^\circ_{GL_n}$   
with infinite upper-triangular $n \times n$-block Toeplitz matrices, with $R$--points
\begin{align}
  \label{eq:41}
  \begin{bmatrix}
    \ddots & \vdots & \vdots       & \vdots       & \vdots     & \vdots & \vdots \\
    \cdots &  1     &  A_1 & A_2 & A_3 & A_4 & \cdots \\
    \cdots &        &  1           &  A_1 & A_2 & A_3 & \cdots \\
    \cdots &        &              &  1      &  A_1 & A_2 & \cdots \\
    \cdots &        &              &         &  1 & A_1 & \cdots \\
    \cdots &        &              &         &    & 1 &  \cdots \\
    \vdots & \vdots & \vdots       & \vdots  & \vdots & \vdots & \ddots\\
  \end{bmatrix}
\end{align}
where all $A_i \in \Mat_{n{\times}n}(R)$ and:
\begin{enumerate}
\item $A_i = 0$ for $i \gg 0$
\item $1 + \sum_{i\geq1} A_i t^{-i}$ is invertible in $\Mat_{n \times n}(R[t])$
\end{enumerate}
 As the restriction map $\SGr^{\circ} \rightarrow \Gr^\circ_{\GL_n}$ factors through $\Gr^\circ_{\GL_1}$ and the map \eqref{eq:42} is injective, we obtain an injective linear map:
\begin{align}
  \label{eq:43}
  \overline{\Sym}  \hookrightarrow \cO ( \Gr^\circ_{GL_n}) 
\end{align}

\subsection{The functions $\det^{(n)}_k$}

For each $k \geq 1$, we define functions $\det^{(n)}_k \in \cO(\Gr^\circ_{\GL_n})$ as follows.
Given an invertible polynomial $A(t^{-1}) = 1 + \sum_{i \geq 1} A_i t^{-i} \in \Gr^\circ_{\GL_n}(R)$, we can consider the determinant $1 + \sum_{i \geq 1} A_i t^{-i} \mapsto \det(A(t^{-1}))$. The functions $\det^{(n)}_k$ are defined to be the coefficients of $\det(A(t^{-1}))$, i.e.  $\det(A(t^{-1}))= 1 + \sum_{k\geq 1} \det^{(n)}_k(A(t)) t^{-k}$.

\begin{Proposition}
  \label{prop:det-in-image-sym}
  For all $s \geq 1$, we have:
  \begin{align}
    \label{eq:47}
\det^{(n)}_k \in \im\big(\overline{\Sym} \hookrightarrow \cO(\Gr^\circ_{\GL_n})\big)
  \end{align}
\end{Proposition}

\begin{proof}

  For a matrix polynomial $A(t^{-1}) = 1 + \sum_{k \geq 1} A_k t^{-k}$, we write  $a^{(k)}_{i,j}$ for the $(i,j)$-th entry of $A_k$  (we set $A_0$ to be the identity matrix). Given integers $k_1, \cdots, k_n \geq 0$, let us form the matrix 
\begin{align*}
    A_{k_1,\cdots,k_n}  = 
  \begin{bmatrix}
   \rule{.6cm}{.03cm} \hspace{.2cm} a^{(k_1)}_{1,\bullet} \hspace{.1cm} \rule{.6cm}{.03cm}  \\
    \vdots   \\
    \rule{.6cm}{.03cm} \hspace{.2cm} a^{(k_n)}_{n,\bullet} \hspace{.1cm} \rule{.6cm}{.03cm}  \\
   \end{bmatrix}
  \end{align*}
  where $ \rule{.6cm}{.03cm} \hspace{.2cm} a^{(k_i)}_{i,\bullet} \hspace{.1cm} \rule{.6cm}{.03cm}$ denotes the $i$-th row of the matrix $A_{k_i}$.
  Then we can verify the following:
  \begin{align*}
\det(A(t^{-1})) = \sum_{k \geq 0} \sum_{k_1 + \cdots + k_n = k} \det(A_{k_1,\cdots,k_n}) t^{-k}
  \end{align*}
  In particular:
  \begin{align*}
    \det^{(n)}_k = \sum_{k_1 + \cdots + k_n = k} \det(A_{k_1,\cdots,k_n}) 
  \end{align*}

  To finish the proof, we prove that $\det(A_{k_1,\cdots,k_n}) \in \im(\Sym \hookrightarrow \cO(\Gr^\circ_{\GL_n}))$.
  Let us form $ e^*_{1 - 1 - k_1 n} \wedge e^*_{2 - 1 - k_2 n} \wedge \cdots e^*_{n-1 - k_n n} \wedge e^*_{n} \wedge e^*_{n+1} \wedge \cdots \in \cF^*$ and consider the image of this element under the inclusion $\cF^* \hookrightarrow \cO(\Gr^0_{\GL_n})$. This element of $\cO(\Gr^0_{\GL_n})$ is given by taking the minor of \eqref{eq:41} corresponding to columns $0,1,2,\cdots$ and rows $1 - 1 - k_1 n,2 - 1 - k_2 n,n-1 - k_n n, n, n+1, \cdots$. This is exactly equal to $\det(A_{k_1,\cdots,k_n})$.

\end{proof}

\subsection{Frobenius twists $h^{(n)}$ and coefficients of the determinant on $\Gr_{GL_n}^\circ$}
\label{sec:frobenius-twists-hn}

\renewcommand{\det}{\text{det}}   
\newcommand{\detnk}{\det^{(n)}_k}

\begin{Theorem}
  \label{thm:shuffle-equals-det}
  As functions in $\cO (\Gr^\circ_{GL_n})$, we have:
  \begin{align}
    \label{eq:005}
     h_k^{(n)} = \detnk
  \end{align}
\end{Theorem}

\begin{proof}
Because of Proposition~\ref{prop:det-in-image-sym} and the fact that the restriction from $\Sym$ to $\cO(\Gr^\circ_{\GL_1})$ is injective (\ref{eq:42}), we can check this equality in $\cO(\Gr_{\GL_1}^\circ)$. Here, we interpret $\detnk$ as a function on infinite matrices of the form~\eqref{eq:39}. Since the functions $h_i \in \cO(\Gr_{\GL_1}^\circ)$ correspond to homogeneous symmetric functions, we can realize \eqref{eq:005} as a certain explicit identity of symmetric functions.

Let us write $E_{i,j}$ for the $n\times n$ matrix with a $1$ in the $(i,j)$-th entry and $0$'s elsewhere. Consider the $n \times n$ matrix-valued polynomial
\begin{align}
  \label{eq:72}
E = t^{-1} E_{n,1} +  \sum_{i=1}^{n-1} E_{i, i+1} 
\end{align}
(this is just the usual principal nilpotent for $\widehat{\mathfrak{sl}_n}$), and form the the following matrix-valued series (with $\cO(\Gr^\circ_{\GL_1})$ coefficients):
\begin{equation}\label{eq:13}
A^{(n)}(t^{-1}) = \sum_{k\geq 0} h_k E^k
\end{equation}
We can also see that $A^{(n)}(t^{-1}) = H(E)$
where $H(z)$ is the formal generating series $\sum_{k\geq 0} h_k z^k \in \cO(\Gr^\circ_{\GL_1})[[z]]$.

Because we can check symmetric function identities using $\mathbb{Q}$-coefficients, we can make use of the following identity relating $H(z)$ with the corresponding series for the power symmetric functions $P(z) = \sum_{k>0} p_k z^{k-1}$:
\begin{align}
  \label{eq:73}
 H(z) = \exp\Big(\int P(z) dz\Big)= \exp\Big( \sum_{k >0} \tfrac{1}{k} p_k z^k\Big)
\end{align}
Combining this with (\ref{eq:13}), it follows that:
\begin{align}
  \label{eq:74}
 \det A^{(n)}(t^{-1}) = \det H(E) = \det \exp\Big(\sum_{k>0} \tfrac{1}{k} p_k E^k \Big) = \exp\Big( \operatorname{tr} \sum_{k>0}\tfrac{1}{k} p_k E^k \Big) 
\end{align}
Observe that 
\begin{align}
  \label{eq:75}
  \operatorname{tr} E^k =
  \begin{cases}
n t^{-j} & \text{if } k = j n, \\ 0 & \text{otherwise}
  \end{cases}
\end{align}
and so
\begin{align}
  \label{eq:76}
 \det H(E) = \exp\Big( \sum_{ j > 0} \tfrac{1}{j} p_{j n} t^{-j} \Big)  = \exp\Big( \sum_{j>0} \tfrac{1}{j} p_j^{(n)} t^{-j} \Big) = \exp\Big(\sum_{j >0} \tfrac{1}{j} p_j t^{-j}\Big)^{(n)} = H(t^{-1})^{(n)}
\end{align}
This proves the claim.
\end{proof}

\subsection{Conclusion of these computations}

\begin{Theorem}
We have an equality of ind-schemes in $\PP(\cF)$:
\begin{align}
 \label{eq:37} 
\Gr_{SL_n} = \PP(V(\Lambda_0)) \cap \Gr_{GL_n}  
\end{align}
\end{Theorem}

\begin{proof}
  By Theorem \ref{thm:KLMW} we have
\begin{align}
  \label{eq:35}
 \Gr_{SL_n} \subseteq \PP(V(\Lambda_0)) \cap \Gr_{GL_n} \subseteq \PP(\cF) 
\end{align}
and the shuffle operators exactly specify those degree-one elements of the homogeneous coordinate ring of $\PP(\cF)$ that vanish on $\Gr_{SL_n}$. Therefore the embeddings are equivariant for $\widehat{SL_n}$. To avoid discussing the central extension of the loop group in an algebro-geometric setting, we note that it is also equivariant for the pro-algebraic group $SL_n[[t]]$. Furthermore, because $SL_n[[t]]$-translates of the big cell of $\Gr_{GL_n}^0$ cover $\Gr_{GL_n}$, it suffices to check that:
\begin{align}
\Gr^{\circ}_{SL_n} = \PP(V(\Lambda_0)) \cap \Gr^{\circ}_{GL_n}  
\end{align}
Specifically, it suffices to show that the shuffle equations imply the $\det=1$ equation defining $\Gr^{\circ}_{SL_n}$ inside $\Gr^{\circ}_{GL_n}$ (see Section \ref{sec:initial-def-of-big-cells}), which follows by Corollary \ref{cor:frobtwists} and Theorem \ref{thm:shuffle-equals-det}.
\end{proof}

Therefore, we have proved the following.
\begin{Proposition}
  \label{prop:simplified-klmw-problem}
  Theorem \ref{thm:klmw-conjecture} holds if and only if $\PP(V(\Lambda_0)) \cap \SGr \subseteq \Gr_{GL_n}$. 
\end{Proposition}
Using this proposition, we will  prove Theorem \ref{thm:klmw-conjecture} by reducing to a problem inside finite-dimensional Grassmannians. In fact we will show that this problem is a special case of a more general problem about finite-dimensional Grassmannians which we will solve.


\section{KP two-tensors, Pl\"ucker equations, and the scheme of invariant subspaces}
\label{sec:kp-two-tensors}

\newcommand{\bfalpha}{\boldsymbol \alpha}
\newcommand{\bfbeta}{\boldsymbol \beta}
\newcommand{\bfgamma}{\boldsymbol \gamma}
\newcommand{\bfdelta}{\boldsymbol \delta}
\newcommand{\tI}{\tilde{I}}

\subsection{KP two-tensors}
\label{sec:KPtwotensors}

Let $d \geq 1$, then we define
\begin{align}
  \label{eq:KP7}
  \Omega_d : \bigwedge^{k} V \otimes \bigwedge^{\ell} V \rightarrow\bigwedge^{k+d} V \otimes \bigwedge^{\ell-d} V 
\end{align}
by the formula:
\begin{align}
\label{eq:8}
\Omega_d(u_{[k]} \otimes v_{[\ell]}) = \sum_{\substack{I \subset [\ell] \\ |I| = d}} (-1)^{I-1} v_I \wedge u_{[k]} \otimes v_{[\ell] \backslash I}  
\end{align}
Recall that $(-1)^{I-1}=(-1)^{i_1-1} \cdots (-1)^{i_d-1}$ where $I=\{i_1,...,i_d\}$.
We can view 
  $
  \Omega_d : \bigwedge^{\bullet} V \otimes \bigwedge^{\bullet} V \rightarrow \bigwedge^{\bullet+d} V \otimes \bigwedge^{\bullet-d} V 
  $
as a bihomogeneous bilinear map of bidegree $(d,-d)$.
In terms of Clifford algebra elements, note that:
\begin{align}
 \Omega_d  
= \sum_{|I|=d} \psi_I  \otimes  \psi^*_I
\end{align}
From this description, we deduce the following proposition that $\Omega_d$ is a divided power of $\Omega_1$.
\begin{Proposition}
\label{prop:dividedpowers}
For any $d$, we have $\Omega_1^d = d! \Omega_d$.
\end{Proposition}

We  also define
\begin{align}
  \label{eq:10}
  \omega_d : \bigwedge^{k}{V} \rightarrow \bigwedge^{k+d}{V} \otimes \bigwedge^{k-d}{V}
\end{align}
by:
\begin{align}
  \label{eq:11}
  \omega_d(\tau) = \Omega_d(\tau \otimes \tau)
\end{align}
We think of $ \omega_d$ as a quadratic vector-valued polynomial on 
$\bigwedge^{k}{V}$. 

\begin{Remark}
The equation $\omega_1(\tau) = 0$ is a finite-dimensional version of the KP hierarchy (cf. \cite[\S 14.11]{K}).
\end{Remark}

\subsection{The Pl\"ucker equations via the KP two-tensors}

Let us fix $k$. Then, for all $d$, because $\omega_d$ is quadratic, we can consider the scheme-theoretic vanishing locus:
\begin{align}
  \label{eq:16}
  \VV(\omega_d) \subset \PP\left(\bigwedge^k V \right)
\end{align}
Similarly, for $k,\ell$  integers with $n=\dim V \geq k \geq \ell \geq 0$, $\Omega_d$ defines a bihomogeneous vector-valued form on $\PP(\bigwedge^{k} V) \times \PP(\bigwedge^{\ell} V)$, and we can consider the corresponding closed subscheme
\begin{align}
 \VV(\Omega_d) \subseteq  \PP\left(\bigwedge^{k} V\right) \times \PP\left(\bigwedge^{\ell} V\right)
\end{align}
This next section is devoted to the proof of the following theorem, which is perhaps known but we could not find it in the literature.

\begin{Theorem}
\label{Theorem:Pluck}
\begin{enumerate}
\item The scheme 
  \begin{align*}
\bigcap_d \VV(\omega_d) \subseteq \PP\left(\bigwedge^k V \right)
  \end{align*}
is equal to $\Gr(k,V)$.
\item The scheme 
\begin{align*}
\bigcap_{d} \VV(\Omega_d) \cap  \left(\Gr(k,V) \times \Gr(\ell,V)\right)  \subseteq  \PP\left(\bigwedge^{k} V\right) \times \PP\left(\bigwedge^{\ell} V\right)
\end{align*}
is  equal to $\Fl_{k,\ell}(V)$.
\item In both cases, the corresponding homogeneous ideal is prime.
\end{enumerate}
\end{Theorem}

\begin{Remark}
Suppose $\Bbbk$ is a field of characteristic zero.  Then it's well known that $\VV(\omega_1)=\Gr(k,V)$ \cite[Exercise 14.27]{K}.  Given the above theorem, this follows immediately from Proposition \ref{prop:dividedpowers}, which also explains why over $\Bbbk$ of arbitrary characteristic we need the higher order KP two-tensors (cf. similar results for the Pl\"ucker ideal \cite{A}).
\end{Remark}

\subsection{The proof of Theorem \ref{Theorem:Pluck}}

To prove this we first need some preparatory lemmas concerning commutations among Clifford operators.  Suppose we have finite sets $I,J$ of integers and $K \subset I\cap J$.  We consider the following sets:
\begin{align*}
\cL(K,J) &= \{ (k,j) \in K \times J \suchthat  j \in J-K \text{ and } k<j\} \\
 \cL(I,J,K) &= I\times J - \{ (i,j) \suchthat i \leq j \text{ and } (i \in K \text{ or } j \in K ) \} 
\end{align*}
and define 
$\sgn(K,J) = (-1)^{|\cL(K,J)|}$ and 
$\sgn(I,J,K) = (-1)^{|\cL(I,J,K)|}$.

\begin{Lemma}
\label{Lemma:Comm1}
The following commutation formulas hold:
\begin{align*}
 \psi_J &= \sgn(K,J) \cdot \psi_{J-K} \psi_K \\
 \psi_I \psi^*_J &= \sum_{K \subset I \cap J} \sgn(I,J,K) \cdot \psi^*_{J-K} \psi_{I-K}  
\end{align*}
\end{Lemma}

\begin{proof}
The first formula follows directly from the Clifford relations.  For the second formula define $\sgn_0(I,J,K)$ by 
$$\psi_I \psi^*_J = \sum_{K \subset I \cap J} \sgn_0(I,J,K) \cdot \psi^*_{J-K} \psi_{I-K}$$
We will show  that $\sgn_0(I,J,K)=\sgn(I,J,K)$.   

Let $I = \{ i_1, \cdots, i_r\}$ where $i_1 < i_2 \cdots < i_r$ as above, and suppose first that $J=\{j\}$.  If $j=i_\ell$ for some $\ell$, then an easy computation shows that 
\begin{align*}
\psi_I\psi_j^*=(-1)^r\psi_j^*\psi_I+(-1)^{r-\ell}\psi_{I-\{j\}}.
\end{align*}  
If $j\notin I$ then $\psi_I\psi_j^*=(-1)^r\psi_j^*\psi_I$.  This proves $\sgn_0(I,J,K)=\sgn(I,J,K)$ in the case when $|J|=1$.

In general let $J=\{j_1, \cdots, j_s\}$ where $j_1 < j_2 \cdots < j_s$, and let $J_0=J-\{j_s\}$.  Then we have the following:
\begin{align*}
\psi_I\psi_J^* =\psi_I\psi_{J_0}^*\psi_{j_s}^* &= \sum_{K_0\subset I\cap J_0}\sgn_0(I,J_0,K_0)\psi_{J_0-K_0}^*\psi_{I-K_0}\psi_{j_s}^* \\
&= \sum_{K_0\subset I\cap J_0}\sgn_0(I,J_0,K_0)\psi_{J_0-K_0}^*\sum_{K_1\subset I\cap\{j_s\}}\sgn_0(I-K_0,j_s,K_1)\psi_{j_s-K_1}^*\psi_{I-K_0-K_1} \\
&= \sum_{K_0\subset I\cap J_0}\sum_{K_1\subset I\cap\{j_s\}}\sgn_0(I,J_0,K_0)\sgn_0(I-K_0,j_s,K_1)\psi_{J-K_0-K_1}^*\psi_{I-K_0-K_1}
\end{align*}
Therefore $\sgn_0$ satisfies the following recursion:
$$
\sgn_0(I,J,K)=\sgn_0(I,J_0,K_0)\sgn_0(I-K_0,j_s,K_1),
$$
where $K_0=K\cap J_0$ and $K_1=K\cap\{j_s\}$.
Observe that $\cL(I,J,K)$ is a disjoint union
\begin{align*}
\cL(I,J,K)=\cL(I,J_0,K_0)\sqcup \cL(I-K_0,j_s,K_1),
\end{align*}
and hence $\sgn$ satisfies the same recursion as $\sgn_0$.  Since they agree when $|J|=1$ this completes the proof.
\end{proof}

\begin{Lemma}
\label{Lemma:prodsgn}
The product $\sgn(J,I,K) \sgn(K,I)$ depends only on $J,K$ and $|I|$.
\end{Lemma}

\begin{proof}
Since $\cL(J,I,K) \cap \cL(K,I)=\emptyset$ it suffices to show that the cardinality of $\cL(J,I,K) \cup \cL(K,I)$ depends only on $J,K$ and $|I|$.  Indeed,
\begin{align*}
\cL(J,I,K) \cup \cL(K,I) &= \{(i,j)\suchthat i>j\} \cup\{(i,j)\suchthat i\leq j \text{ and } j\notin K \} \\
&= \{(i,j)\suchthat i>j \text{ and } j\in K\} \cup \{(i,j)\suchthat j\notin K \}
\end{align*}
Clearly $|\{(i,j)\suchthat i>j \text{ and } j\in K\}|$ can be deduced from $J,K$, and $|\{(i,j)\suchthat j\notin K \}|=|J|(|I|-|K|)$.
\end{proof}
Using the above lemma, we define $\varepsilon_d(J,K)=\sgn(J,I,K) \sgn(K,I)$, where $I$ is any subset containing $K$ and $d=|I|$.  
Let $\bfalpha = (\alpha_1,...,\alpha_k )$ and $\bfbeta = (\beta_1,...,\beta_\ell)$ be increasing sequences with $\alpha_i,\beta_i\in [n]$, and $k\geq\ell$.  Let $X_{\bfalpha} = e^*_\bfalpha \in \bigwedge^k V^*$ and $X_{\bfbeta} = e^*_\bfbeta \in \bigwedge^\ell V^*$ be the corresponding Pl\"ucker coordinates.  For  $1 \leq d \leq \ell$ consider the function given by
\begin{align}
\label{Eq:Pluck}
 P_{\bfalpha \otimes \bfbeta,d}=X_{\bfalpha} \otimes X_{\bfbeta} - \sum_{0 < t_1 < \cdots < t_d\leq k} X_{\alpha_1,..., \beta_1,..., \beta_d,...\alpha_k} \otimes X_{\alpha_{t_1},...,\alpha_{t_d}, \beta_{d+1},...,\beta_\ell}  \in \bigwedge^k V^* \otimes \bigwedge^\ell V^*
\end{align}
In this sum, the $\beta_1,..., \beta_d$ replace the $\alpha_{t_1},..., \alpha_{t_d}$ and vice-versa.
We consider $P_{\bfalpha \otimes \bfbeta,d}$ as a bihomogeneous form on $\PP(\bigwedge^kV)\times\PP(\bigwedge^\ell V)$ of bidegree $(1,1)$. Therefore, we can consider its vanishing locus $\VV(P_{\bfalpha \otimes \bfbeta,d})$, which is a closed subscheme of $\PP(\bigwedge^kV)\times\PP(\bigwedge^\ell V)$.

When $k=\ell$ we can also form:
\begin{align}
  \label{eq:58}
 P_{\bfalpha,\bfbeta,d} =X_{\bfalpha}  X_{\bfbeta} - \sum_{0 < t_1 < \cdots < t_d\leq k} X_{\alpha_1,..., \beta_1,..., \beta_d,...\alpha_k}  X_{\alpha_{t_1},...,\alpha_{t_d}, \beta_{d+1},...,\beta_k}  \in \Sym^2\left(\bigwedge^k V^*\right) 
\end{align}
We consider $P_{\bfalpha,\bfbeta,d}$ as a degree-$2$ element of the homogenous coordinate ring of $\PP(\bigwedge^kV)$, and therefore we can consider its vanishing locus $\VV(P_{\bfalpha,\bfbeta,d}) \subseteq \PP(\bigwedge^kV)$.

\begin{Proposition}[\cite{F}, Section 9.1]
\label{prop: Plucker relations for Gr and incidence}
\begin{enumerate}
\item The Grassmannian of $k$-planes $\Gr(k,V)$ is equal to the following closed subscheme of $\PP(\bigwedge^kV)$:
  \begin{align}
    \label{eq:59}
 \Gr(k,V) =  \bigcap_{\bfalpha,\bfbeta} \bigcap_{d\geq 1} \VV(P_{\bfalpha,\bfbeta,d}) \subseteq \PP(\bigwedge^kV)
  \end{align}
In the intersection above, $\bfalpha,\bfbeta$ are both of length $k$.

\item The incidence variety $\Fl_{k,\ell}(V)$ is equal to the
  following closed subscheme
  of $\PP(\bigwedge^kV)\times\PP(\bigwedge^\ell V)$:
  \begin{align}
    \label{eq:60}
   \Fl_{k,\ell}(V) = \bigcap_{\bfalpha,\bfbeta} \bigcap_{d\geq 1} \VV(P_{\bfalpha \otimes \bfbeta,d}) \cap \Gr(k,V)\times\Gr(\ell,V)
\subseteq \PP(\bigwedge^kV)\times\PP(\bigwedge^\ell V)
  \end{align}
In the intersection above, $\bfalpha$ is of length $k$, and $\bfbeta$ is of length $\ell$.

\item In both cases, the corresponding homogeneous ideal is prime.
\end{enumerate}
\end{Proposition}

 The Clifford operators define maps
\begin{align*}
\psi_i:\bigwedge^kV^* \to \bigwedge^{k+1}V^*, \quad
\psi_i^*:\bigwedge^kV^* \to \bigwedge^{k-1}V^*
\end{align*}  
We abuse notation slightly, since here $\psi_i$ (respectively $\psi_i^*$) is defined by wedging with $e_i^*$ (respectively taking the interior product with $e_i$).  These extend to maps $\psi_I,\psi_I^*$, and we let $\Omega_d^*=\sum_{|I|=d}\psi_I^*\otimes\psi_I$, thought of as a map $\Omega_d^*: \bigwedge^{k+d} V^* \otimes \bigwedge^{\ell-d} V^* \rightarrow\bigwedge^{k} V^* \otimes \bigwedge^{\ell} V^*$.  Note that $\Omega_d^*$ is the adjoint of $\Omega_d$ defined in (\ref{eq:KP7}).  

Now we can reformulate (\ref{Eq:Pluck}) as:
\begin{align}
 P_{\bfalpha\otimes \bfbeta,d}=X_{\bfalpha} \otimes X_{\bfbeta} - \sum_{|I| = d} \psi_J\psi^*_{I} X_{\bfalpha} \otimes \psi_{I} \psi^*_J X_{\bfbeta}
\end{align}
where in this sum $I$ ranges over $d$ element subsets of $[n]$, and $J = \{\beta_1,..., \beta_d\}$.
The $I$-term will be non-zero only if $I \subset \bfalpha$. 
The following computation will be key to proving Theorem \ref{Theorem:Pluck}.
\begin{Lemma}
\label{Lemma:rhsPluck}
Fix $k\geq\ell$.  Let $\bfalpha = (\alpha_1,...,\alpha_k )$ and $\bfbeta = (\beta_1,...,\beta_\ell)$ be increasing sequences with $\alpha_i,\beta_i\in [n]$.  Fix $1\leq d\leq \ell$ and let $J=(\beta_1,...,\beta_d)$. Then we have that
\begin{align*}
  \sum_{|I| = d} \psi_{J} \psi^*_{I} X_{\bfalpha} \otimes \psi_{I} \psi^*_{J} X_{\bfbeta}=\sum_{K \subset J} \varepsilon_d(J,K) \Omega^*_{d-|K|} ( \psi_{J-K} X_{\bfalpha} \otimes \psi_K \psi^*_J X_{\bfbeta} )
\end{align*}
\end{Lemma}

\begin{proof}
By Lemma  \ref{Lemma:Comm1} we have:
\begin{align*} 
\sum_{|I| = d} \psi_{J} \psi^*_{I} X_{\bfalpha} \otimes \psi_{I} \psi^*_{J} X_{\bfbeta} &=
\sum_{|I| = d} \sum_{K \subset I \cap J} \sgn(J,I,K) \cdot \psi_{I-K}^* \psi_{J-K} X_{\bfalpha} \otimes \psi_{I} \psi^*_{J} X_{\bfbeta} \\
&= \sum_{|I| = d} \sum_{K \subset I \cap J} \sgn(J,I,K) \sgn(K,I) \cdot \psi_{I-K}^* \psi_{J-K} X_{\bfalpha} \otimes \psi_{I-K} \psi_K \psi^*_{J} X_{\bfbeta}  
\end{align*}
Because $K$ must be a subset of $J$, we can substitute $\tI = I - K$ and rewrite this sum as follows: 
\begin{align*}
\sum_{K \subset J} \sum\limits_{\substack{|\tI| = d - |K| \\ \tI \cap K = \emptyset}}  \sgn(J,I,K) \sgn(K,I) \cdot \psi_{\tI}^* \psi_{J-K} X_{\bfalpha} \otimes \psi_{\tI} \psi_K \psi^*_{J} X_{\bfbeta}
\end{align*}
By Lemma \ref{Lemma:prodsgn} this is equal to 
\begin{align*}
\sum_{K \subset J} \sum\limits_{\substack{|\tI| = d - |K| \\ \tI \cap K = \emptyset}}  \varepsilon_d(J,K) \cdot \psi_{\tI}^* \psi_{J-K}X_{\bfalpha} \otimes \psi_{\tI} \psi_K \psi^*_{J} X_{\bfbeta}
\end{align*}
Note that if $\tI \cap K \neq \emptyset$, then $\psi_{\tI} \psi_{K} = 0$, so we can drop the condition of $\tI \cap K = \emptyset$ in the above sum.  It's easy to see now that this agrees with 
\begin{align*}
\sum_{K \subset J} \varepsilon_d(J,K) \Omega^*_{d-|K|} ( \psi_{J-K} X_{\bfalpha} \otimes \psi_K \psi^*_J X_{\bfbeta} )
\end{align*}
\end{proof}

\begin{proof}[Proof of Theorem \ref{Theorem:Pluck}]
In cases (1), (2) of the theorem, we will show that the ideal defined by KP two-tensors is equal to the corresponding Pl\"ucker ideal.  The three parts of Proposition \ref{prop: Plucker relations for Gr and incidence} then imply the three parts of the theorem.

First we focus on case (1).  Let $R=\bigoplus_{m\geq0}R_m$ be the homogenous coordinate ring of $\PP(\bigwedge^kV)$. 
Let $\cI$ be the homogenous ideal of $R$ generated by $P_{\bfalpha,\bfbeta,d}$, where  $\bfalpha $ and $\bfbeta$ are increasing sequences of length $k$ with entries in $[n]$, and $1\leq d\leq k$.  Let $\cI'$ be the homogenous ideal of $R$ generated by $(\kappa \otimes \lambda)\circ \omega_d$, where $\kappa \in \bigwedge^{k+d}V^*$, $\lambda\in \bigwedge^{k-d}V^*$, and $1\leq d\leq k$.  We wish to show that $\cI=\cI'$.  

By Lemma \ref{Lemma:rhsPluck} we have that
\begin{align*}
P_{\bfalpha,\bfbeta,d}=X_{\bfalpha} X_{\bfbeta} - \sum_{K \subset J} \varepsilon_d(J,K) m\left(\Omega^*_{d-|K|} ( \psi_{J-K} X_{\bfalpha} \otimes \psi_K \psi^*_J X_{\bfbeta} )\right)
\end{align*}
where  $m:R\otimes R \to R$ is the multiplication map.  

Note that $\varepsilon_d(J,J)=(-1)^{\binom{d}{2}}$ and $\psi_J \psi^*_J X_{\bfbeta}=(-1)^{\binom{d}{2}}X_{\bfbeta}$.  Since $\Omega^*_0$ is the identity operator, we have that
\begin{align*}
\varepsilon_d(J,J) m\left(\Omega^*_{0} ( X_{\bfalpha} \otimes \psi_J \psi^*_J X_{\bfbeta} )\right)=X_{\bfalpha}X_{\bfbeta}
\end{align*}
and therefore we can express $P_{\bfalpha,\bfbeta,d}$ as a sum over strict subsets of $J$:
\begin{align*}
P_{\bfalpha,\bfbeta,d}=\sum_{K \subsetneq J} \varepsilon_d(J,K) m\left(\Omega^*_{d-|K|} ( \psi_{J-K} X_{\bfalpha} \otimes \psi_K \psi^*_J X_{\bfbeta} )\right)
\end{align*}
For $K \subsetneq J$, let $\kappa_K=\psi_{J-K} X_{\bfalpha}\in \bigwedge^{k+d-|K|}V^*$ and $\lambda_K=\psi_K \psi^*_J X_{\bfbeta}\in \bigwedge^{k-d+|K|}V^*$.  The map $$m\circ \Omega^*_{d-|K|}:\bigwedge^{k+d-|K|}V^*\otimes \bigwedge^{k-d+|K|}V^* \to R_2$$ satisfies 
\begin{align*}
m\circ \Omega^*_{d-|K|}(\kappa_K \otimes \lambda_K) = (\kappa_K \otimes \lambda_K)\circ \omega_{d-|K|}
\end{align*}
Therefore 
\begin{align}
\label{Eq:Pluckomega}
P_{\bfalpha,\bfbeta,d}=\sum_{K \subsetneq J} \varepsilon_d(J,K)(\kappa_K \otimes \lambda_K)\circ \omega_{d-|K|}
\end{align}
Equation (\ref{Eq:Pluckomega}) immediately implies that $\cI\subset \cI'$.  

It is possible to prove the opposite containment $\cI \supset \cI'$ algebraically, analogously to the above.  Instead, we opt for a geometric argument to prove equality: we will show that the vanishing loci of these ideals have the same field points, for any field.  

Note that both ideals are defined over $\ZZ$ (i.e. we may take our coefficient ring $\Bbbk = \ZZ$); we will work over $\ZZ$ and by abuse of notation we will continue to write $\cI, \cI'$.  Thus $\cI$ defines $\Gr(k, V)\subset \PP(\bigwedge^k V)$ as a scheme over $\Spec \ZZ$. Denote its cone 
$$ 
X = \operatorname{Cone}\big(\Gr(k,V) \big) \subset \bigwedge^k V 
$$
We claim that for any field $\FF$, the vector-valued functions $\omega_d$ vanish on all $\FF$--points $X(\FF)$.  Indeed, points $\tau \in X(\FF)$ are precisely the pure tensors $\tau = u_1\wedge\cdots \wedge u_k$, on which the vanishing of $\omega_d$ is easily verified.
By the Nullstellensatz \cite[Theorem IX.1.5]{Lang}, it follows that $\omega_d$ is in the defining ideal of $X(\FF)$, as a subvariety of $\bigwedge^k V$ over $\FF$.  In other words, $\omega_d$ is in the Pl\"ucker ideal $\cI \otimes_\ZZ \FF$.  This implies that
\begin{equation}
\label{eq: field points Plucker}
\cI' \otimes_\ZZ \FF \subset \cI \otimes_\ZZ \FF
\end{equation}

We have shown above that $\cI \subset \cI'$.  This is an inclusion of homogeneous ideals, and the inclusions of homogeneous components $\cI_r \subset \cI'_r$ are of finitely-generated abelian groups, for all $r\geq 0$.  Equation (\ref{eq: field points Plucker}) implies that there is an equality
\begin{equation}
\cI'_r \otimes_\ZZ \FF = \cI_r \otimes_\ZZ \FF,
\end{equation}
fo any field $\FF$.  In particular, this holds for $\FF = \QQ$ and $\FF = \ZZ / p$ for any prime $p$, and so we deduce equality of our finitely-generated abelian groups $\cI_r = \cI'_r$.  Therefore $\cI = \cI'$.

%
%
%
The proof of part (2) of the theorem follows similarly.  We let $R$ denote now the bihomogenous coordinate ring of $\Gr(k,V)\times\Gr(\ell,V)$.  Let $\cI$ be the bihomogenous ideal of $R$ generated by $P_{\bfalpha,\bfbeta,d}$, where  $\bfalpha $ and $\bfbeta$ are increasing sequences of length $k$ and $\ell$ with entries in $[n]$, and $1\leq d\leq \ell$.  Let $\cI'$ be the bihomogenous ideal of $R$ generated by $(\kappa \otimes \lambda)\circ \Omega_d$, where $\kappa \in \bigwedge^{k+d}V^*$, $\lambda\in \bigwedge^{\ell-d}V^*$, and $1\leq d\leq \ell$.  

By the same reasoning as in the proof of part (1), we can write 
\begin{align}
P_{\bfalpha\otimes\bfbeta,d}=\sum_{K \subsetneq J} \varepsilon_d(J,K)(\kappa_K \otimes \lambda_K)\circ \Omega_{d-|K|}
\end{align}
This immediately implies $\cI \subset \cI'$, and a similar argument using the Nullstellensatz shows that $\cI'=\cI$.
\end{proof}

\subsection{The scheme of $T$-invariant subspaces}

Let $T$ be a invertible operator in $\GL(V)$. Consider
\begin{align}
  \leftexp{\prime}{\cG^T} = \Big\{ U \in \Gr(k,V) \suchthat TU = U \Big\},
\end{align}
which defines a moduli functor. That is, we view $U$ as a subbundle of the trivial bundle with fiber $V$ over a test scheme $S$, and the same for $TU$.

Define 
\begin{align}
  \Omega^T_d : \bigwedge^{k} V \otimes \bigwedge^{\ell} V \rightarrow\bigwedge^{k+d} V \otimes \bigwedge^{\ell-d} V 
\end{align}
by:
\begin{align}
  \Omega^T_d(u_{[k]} \otimes  v_{[\ell]}) 
&= \sum_{|I|=d} (-1)^{I}(T v_{[I]}) \wedge u_{[k]} \otimes v_{[\ell]\setminus I} \\
&= \sum_{|I|=d} (-1)^{I}e_I \wedge u_{[k]} \otimes  \iota_{T^*e^*_I}(v_{[\ell]}) \\
&=\sum_{|I|=d} (-1)^{I}(T e_I) \wedge u_{[k]} \otimes  \iota_{e^*_I}(v_{[\ell]})
\end{align}
(See Section \ref{sec:KPtwotensors} for the definition of $(-1)^I$.)
Similarly we set
\begin{align}
  \omega^T_d : \bigwedge^{k}{V} \rightarrow \bigwedge^{k+d}{V} \otimes \bigwedge^{k-d}{V}
\end{align}
by:
\begin{align}
  \omega^T_d(\tau) = \Omega^T_d(\tau \otimes \tau)
\end{align}
Again, we think of this as a quadratic vector-valued polnomial on $\bigwedge^{k}{V}$.

Let us also define
\begin{align}
 \eta_d^T : \bigwedge^{k}{V} \rightarrow \bigwedge^{k+d}{V} \otimes \bigwedge^{k-d}{V}
\end{align}
by
\begin{align}
  \eta^T_d(\tau) = \Omega_d(T\tau \otimes \tau)
\end{align}

\begin{Theorem}
\label{thm: equations for 'G}
Let $T:V\to V$ be an invertible operator.  Then
  \begin{align*}
 \leftexp{\prime}{\cG^T} = \bigcap_{d}\VV(\omega_d^T) \cap \Gr(k,V) \subseteq \PP\left( \bigwedge^k V \right) 
  \end{align*}
\end{Theorem}

\begin{proof}

First, we claim that 
\begin{equation}
\leftexp{\prime}{\cG^T} = \bigcap_d \VV(\eta_d^T) \cap \Gr(k, V)
\end{equation}
Indeed, let 
$$
\cF^T_{k,\ell}(V)=\{(U,W)\in \Gr(k, V)\times\Gr(\ell,V)\suchthat TU\subset W\}
$$
Theorem \ref{Theorem:Pluck}(2) implies that 
$$
\cF^T_{k,\ell}(V)=\bigcap_d \VV(\Omega_d\circ(T\otimes 1)) \cap (\Gr(k, V)\times\Gr(\ell,V))
$$
Now fix $k=\ell$.  Consider the diagonal map $\Delta:\Gr(k,V) \to \Gr(k,V) \times \Gr(k,V)$.  Then
\begin{align*}
\leftexp{\prime}{\cG^T} &=  \bigcap_d \VV(\Omega_d\circ(T\otimes 1)) \cap \Delta(\Gr(k,V)) \\
&= \bigcap_d \Delta(\VV(\eta_d^T)) \cap \Delta(\Gr(k,V)) \\
&= \bigcap_d \VV(\eta_d^T) \cap \Gr(k,V)
\end{align*}   

  Next, by definition $\leftexp{\prime}{\cG^T} = \leftexp{\prime}{\cG^{T^{-1}}}$.  By the claim, we thus have
  \begin{align*}
  \leftexp{\prime}{\cG^{T^{-1}}} = \bigcap_{d} \VV(\eta_d^{T^{-1}}) \cap \Gr(k,V)  
  \end{align*}
  We compute
  \begin{align*}
   \omega_d^T( v_{[k]})  = 
    \Omega_d^T( v_{[k]} \otimes v_{[k]}) 
	&= \sum_{\substack{I \subset [k] \\ |I|=d}} (-1)^{I} Tv_I \wedge v_{[k]} \otimes v_{[k]\setminus I}   \\
	&=    (T \otimes 1) \left(  \sum_{\substack{I \subset [k] \\ |I|=d}} (-1)^{I} v_I \wedge T^{-1}v_{[k]} \otimes v_{[k]\setminus I}  \right) \\
	&= (T \otimes 1) \circ \Omega_d ( T^{-1} v_{[k]} \otimes v_{[k]}) \\
	&= (T \otimes 1) \eta^{T^{-1}}_d (v_{[k]})
  \end{align*}
  Because $T\otimes1$ is invertible, we conclude:
  \begin{align*}
   \VV(\omega_d^T) = \VV(\eta_d^{T^{-1}})
  \end{align*}
\end{proof}

\subsection{A lemma about two-tensors}

\begin{Lemma}
  \label{lem:t-one-plus-t-two-two-tensor}
Suppose $T_1, T_2$ are linear operators on $V$.  Then as operators on $\bigwedge^k V\otimes \bigwedge^\ell V$, we have:
\begin{align*}
\Omega_d^{T_1 + T_2} = \sum_{k=0}^d \Omega_k^{T_1} \circ \Omega_{d-k}^{T_2} 
\end{align*}
\end{Lemma}
\begin{proof}
We have
\begin{align*}
  \Omega_d^{T_1 + T_2} & = \sum_{|I| = d} e_I \otimes \iota_{(T_1+T_2)^* e^*_I} \\
  &= \sum_{|I| = d } e_I \otimes \iota_{(T_1+ T_2)^\ast e_{i_1}^\ast} \cdots \iota_{(T_1+T_2)^\ast e_{i_d}^\ast} \\
  &= \sum_{|I| = d} e_I \otimes \left( \iota_{T_1^\ast e_{i_1}^\ast} + \iota_{T_2^\ast e_{i_1}^\ast}\right) \cdots \left( \iota_{T_1^\ast e_{i_d}^\ast} + \iota_{T_2^\ast e_{i_d}^\ast}\right) \\
  &= \sum_{|I| = d} \sum_{\substack{K, L \\ I = K\sqcup L}} e_K e_L \otimes \iota_{T_1^\ast e_K^\ast} \iota_{T_2^\ast e_L^\ast} \\
  &= \sum_{\substack{|K| = k, |L| = \ell \\ k+\ell = d}} e_K e_L \otimes \iota_{T_1^\ast e_K^\ast} \iota_{T_2^\ast e_L^\ast} \\
  &= \sum_{k+\ell = d} \Omega_k^{T_1} \circ \Omega_{\ell}^{T_2}
\end{align*}
Here, the first three lines simply apply the definition of $\iota$.  To pass from the third line to the fourth, we have used the fact that the $e_{i}$ (resp. $\iota$) anticommute; all signs cancel.  For the fifth line we used that $e_K e_L$ is zero unless $K\cap L = \emptyset$.
\end{proof}

\subsection{Another version of the scheme of $T$-invariant subspaces}

Let us fix a linear operator $T : V \rightarrow V$, which we no longer require to be invertible. We can form the closed scheme $\cG^T \subseteq \Gr(k,V)$ consisting of $T$-invariant $k$-planes. Note, that this subscheme will usually not be reduced (cf. Example \ref{example:reduced}). Precisely, we have the following. 

\begin{Definition}
For any test scheme $S$, we define $\cG^T(S)$ to be the set of rank-$k$ vector sub-bundles $\cE \hookrightarrow \underline{V}$ on $S$, where $\underline{V}$ is the trivial vector bundle on $S$ with fiber $\underline{V}$, such that the composed map of sheaves on $S$
\begin{align}
  \label{eq:32}
  \cE \injects \underline{V} \xrightarrow{T} \underline{V} \surjects \underline{V}/\cE
\end{align}
 is zero. 
\end{Definition}

Let us write $\cT$ for the tautological bundle on $\Gr(k,V)$. Then $T$ defines an element of $\hom(\cT,\underline{V}/\cT)$. We can identify the sheaf $\cH\text{om}(\cT,\underline{V}/\cT)$ with the tangent bundle, and therefore, we can identify $\hom(\cT,\underline{V}/\cT)$ with global vector fields. Therefore, to $T$ we associate a global vector field, and $\cG^T$ is precisely the vanishing locus of that vector field.

\begin{Proposition}
  Suppose $T$ is invertible. Then
  \begin{align*}
   \cG^T = \leftexp{\prime}\cG^T
  \end{align*}
\end{Proposition}
\begin{proof}
We will show that the two moduli functors are equal.  It is clear that $\cG^T \supset \leftexp{\prime}\cG^T$. 

For the opposite inclusion, suppose that $\cE \in \cG^T(S)$.   To show that $T \cE = \cE$, it suffices to work at the level of stalks.  In terms of stalks at $P\in S$, the composition (\ref{eq:32}) is of free $\cO_{S, P}$--modules:
\begin{equation}
\cE_P \hookrightarrow \underline{V}_{P} \xrightarrow{T} \underline{V}_{P}  \surjects  \underline{V}_{P} / \cE_P
\end{equation}
In particular since $T$ is invertible, it follows that $T: \cE_P\hookrightarrow \cE_P$.  Modulo the maximal ideal $\mathfrak{m}_P$ of $\cO_{S,P}$, we see that $T$ defines an injective endomorphism of the finite-dimensional vector space $\cE_P / \mathfrak{m}_P$, so an isomorphism. Therefore by Nakayama's lemma $T:\cE_P \stackrel{\sim}{\rightarrow} \cE_P$.  Hence $T \cE = \cE$, which proves the claim.
\end{proof}

\begin{Theorem}
 \label{thm:equations-for-G-T}
  Suppose $T$ is nilpotent. Then
  \begin{align}
    \cG^T = \bigcap_{d}\VV(\omega_d^T) \cap \Gr(k,V)
  \end{align}
\end{Theorem}

\begin{proof}
  One can check using \eqref{eq:32} that for any $T$,
    $
   \cG^{T} =  \cG^{I+T}
    $.
  Because $T$ is nilpotent, $I+T$ is invertible. Therefore, we have 
  $
   \cG^{T} =  \leftexp{\prime}{\cG^{I+T}}
   $,
  and by Theorem \ref{thm: equations for 'G},
  $
   \cG^{T} =  \bigcap_{d}\VV(\omega_d^{I+T}) \cap \Gr(k,V)
   $.
  By Theorem \ref{Theorem:Pluck}(1), we have:
  \begin{align*}
   \cG^{T} =  \bigcap_{d}\VV(\omega_d^{I+T}) \cap   \bigcap_{d}\VV(\omega_d^{I}) \subseteq \PP\left(\bigwedge^{k}V \right)
  \end{align*}
  Finally by Lemma \ref{lem:t-one-plus-t-two-two-tensor}, we have:
  \begin{align*}
\bigcap_{d}\VV(\omega_d^{I+T}) \cap   \bigcap_{d}\VV(\omega_d^{I}) =     \bigcap_{d}\VV(\omega_d^{T}) \cap   \bigcap_{d}\VV(\omega_d^{I})
  \end{align*}
\end{proof}


\section{Shuffle equations in the finite-dimensional setting}

\subsection{Generalized shuffle operators}

Let $T:V \rightarrow V$ be a linear operator.  For all $k$, we define
\begin{align}
  \label{eq:1}
 sh_d^T : \bigwedge^k V \rightarrow \bigwedge^k V 
\end{align}
by the formula:
 \begin{align}
   \label{eq:28}
 sh_d^T(u_1 \wedge \cdots \wedge u_k) = \sum_{R \subset [k] : |R| = d} T_R( u_1 \wedge \cdots \wedge u_k) 
 \end{align}
\newcommand{\textif}{\text{ if }}
where $T_R(u_1\wedge \cdots \wedge u_k) = v_1 \wedge\cdots \wedge v_k$ and 
\begin{align}
  \label{eq:30}
  v_r = 
  \begin{cases}
    Tu_r & \textif r \in R \\
    u_r & \textif r \notin R 
  \end{cases}
\end{align}

\begin{Lemma}
  \label{prop:explicit-basis-free-shuffle-formula}
We have the following  formula:
\begin{align*}
 sh_d^T(u_1 \wedge \cdots \wedge u_k) = \sum_{|I|=d}e_I \wedge \iota_{T^*(e_I^*)}  (u_1 \wedge \cdots \wedge u_k)
\end{align*}

\end{Lemma}

We can put this all together to think of 
$
 sh_d^T : \bigwedge^\bullet V \rightarrow \bigwedge^\bullet V 
 $
as a homogeneous linear operator of degree $0$.

The shuffle operators have the following alternate description, which we will make use of to prove Proposition \ref{prop:Fpoints}:
\begin{Lemma}
\label{lemma: alternate description of shuffles}
For any $\tau \in \wedge^k V$ and $t \in R$, we have 
\begin{align*}
(I + t T) (\tau) = \tau + \sum_{d=1}^k t^d sh_d^T(\tau)
\end{align*}
\end{Lemma}
\begin{proof}
By linearity, we may assume that $\tau = u_1 \wedge \cdots \wedge u_k$.  By definition,
$$ (I + t T) (\tau) = \big(u_1 + t T(u_1) \big) \wedge \cdots \wedge \big(u_k + t T(u_k) \big) $$
Expanding the right-hand side in powers of $t$, we get the claim.
\end{proof}

\begin{Remark}
\label{rmk: operators from sym in fd case}
We can generalize the definition of $sh_d^T$: let $f(x_1,\ldots,x_k) \in \ZZ[x_1,\ldots,x_k]^{S_k}$ be a symmetric polynomial in $k$ variables (we use ground ring $\ZZ$ for simplicity).  Then we may define an associated endomorphism of $\bigwedge^k V$ as follows.  First, consider the endomorphism of $V^{\otimes k}$ defined by:
$$ T_i = \operatorname{Id}_V^{\otimes (i-1)} \otimes T \otimes \operatorname{Id}_V^{\otimes(k-i)} $$
Then the operator $f(T_1, T_2,\ldots, T_k)$ on $V^{\otimes k}$ descends to an endomorphism of $\bigwedge^k V$, as desired. So we get a map $\ZZ[x_1,\ldots,x_k]^{S_k} \rightarrow \End(\wedge^k V)$.  In particular, the operator associated to $e_d(x_1,\ldots, x_k) \in \ZZ[x_1,\ldots,x_k]^{S_k}$ is precisely $sh_d^T$.  This is the finite-dimensional analog of the discussion in Section \ref{section: Shuffle operators and Frobenius twists}.
\end{Remark}

\subsection{The scheme $\cS^T$}

Define
  \begin{align}
    \label{eq:19}
   \cS^T = \bigcap_d \VV(sh^T_d) \subseteq \Gr(k,V)
  \end{align}
\begin{Theorem}
\label{thm:FinAnalog}
  Suppose $T$ is nilpotent. Then:
  \begin{align*}
     \cS^T \subseteq \cG^T 
  \end{align*}
\end{Theorem}
This follows immediately from Theorem \ref{thm:equations-for-G-T} and the following formula.
\begin{Proposition}
\begin{align*}
  \omega^T_d(\tau) = \sum_{|I|=d} e_I \wedge (sh^T_d \tau) \otimes \iota_{e^*_I}(\tau)
\end{align*}
\end{Proposition}

\begin{proof}
 Let $\tau = u_1 \wedge \cdots \wedge u_k $. By definition, we have:
 \begin{align*}
   \omega_d^T(\tau) = \sum_{|I|=d} e_I \wedge u_1 \wedge \cdots \wedge u_k \otimes \iota_{T^*e^*_I} (u_1 \wedge \cdots \wedge u_k) = \\
\sum_{|I|=d} e_I \wedge u_1 \wedge \cdots \wedge u_k \otimes \sum_{\substack{R \subset [k] \\ |R| = d}} (-1)^{R-1} \langle e^*_I, Tu_{R} \rangle u_{[k] \backslash R}
 \end{align*}
Recall that if $R = \{ r_1, \cdots, r_d \}$, we write $(-1)^{R-1} = (-1)^{r_1-1} \cdots (-1)^{r_d-1}$.
The above expression is equal to:
\begin{align}
  \label{eq:23}
 \sum_{\substack{R \subset [k] \\ |R| = d}} (-1)^{R-1} Tu_R \wedge u_1 \wedge \cdots \wedge u_k \otimes  u_{[k] \backslash R}
\end{align}
Now, let us consider
\begin{align*}
\sum_{|I|=d} e_I \wedge (sh^T_d \tau) \otimes \iota_{e^*_I}(\tau)
\end{align*}
Using (\ref{eq:28}), this is equal to:
\begin{align*}
\sum_{|I|=d} e_I \wedge \sum_{\substack{R \subset [k] \\ |R| = d}}T_R (u_1 \wedge \cdots \wedge u_k) \otimes \sum_{\substack{S \subset [k] \\ |S| = d}} (-1)^{S-1} \langle e^*_I, u_S \rangle  u_{[k] \backslash S}  
\end{align*}
Summing over $I$, we get:
\begin{align}
  \label{eq:sum-over-I}
\sum_{\substack{R \subset [k] \\ |R| = d}}\sum_{\substack{S \subset [k] \\ |S| = d}} (-1)^{S-1} u_S \wedge T_R( u_1 \wedge \cdots \wedge u_k) \otimes u_{[k] \backslash S}  
\end{align}
In the sum \eqref{eq:sum-over-I}, the term $u_S \wedge T_R( u_1 \wedge \cdots \wedge u_k)$ is equal to zero unless $S = R$. If not, a factor of $u_s$ would appear twice for any $s \in S \backslash R$.

So \eqref{eq:sum-over-I} is equal to 
\begin{align}
  \label{eq:27}
\sum_{\substack{R \subset [k] \\ |R| = d}} (-1)^{R-1} u_R \wedge T_R( u_1 \wedge \cdots \wedge u_k) \otimes u_{[k] \backslash R}  
\end{align}
which is equal to~\eqref{eq:23}.
\end{proof}

If we look at field points, we also have the opposite inclusion to that of Theorem \ref{thm:FinAnalog}.

\begin{Proposition}
\label{prop:Fpoints}
Suppose $T$ is nilpotent.  Then we have equality of $\FF$--points 
\begin{align*}
\cS^T(\FF) = \cG^T(\FF)
  \end{align*}
for any field $\FF$.
\end{Proposition}

\begin{proof}

As $\cG^T$ is a closed subscheme of $\PP(\wedge^k V)$, we can consider the cone on $\cG^T$, $\operatorname{Cone}(\cG^T)$, which is a closed subscheme of $\wedge^k V$.
Consider $X = (\operatorname{Cone}(\cG^T) \times \AA^1)_{red}$,  the closed subscheme of $\wedge^k V \times \AA^1$ that is $\operatorname{Cone}(\cG^T) \times \AA^1$ with its induced reduced scheme structure. In particular, for any field extension $\FF \subset \mathbb{L}$, we have $X(\mathbb{L}) = (\operatorname{Cone}(\cG^T)\times \AA^1) (\mathbb{L})$.  We will show that 
\begin{equation}
\label{eq: ST = GT}
\tau = ( I + t T)(\tau)
\end{equation}
for any $(\tau, t) \in X(\overline{\FF})$, where $\overline{\FF}$ denotes an algebraic closure of $\FF$.  Assuming this claim, it follows from the Nullstellensatz \cite[Theorem IX.1.5]{Lang} that the difference between the two sides of (\ref{eq: ST = GT}), thought of as a vector-valued function on $X$, is in the ideal defining $X$ in the coordinate ring of $\wedge^k V \times \AA^1$.  Since $X$ is the product of $\operatorname{Cone}(\cG^T)_{red}$ with $\AA^1$, from Lemma \ref{lemma: alternate description of shuffles} we see that all $sh_d^T$ must lie in the ideal defining $\operatorname{Cone}(\cG^T)_{red}$.  This proves that $\cS^T(\FF) \supseteq (\cG^T)_{red}(\FF) =  \cG^T(\FF)$. The reverse inclusion is Theorem \ref{thm:FinAnalog}.

To prove (\ref{eq: ST = GT}), consider any $(\tau, t)\in X(\overline{\FF})$. We may assume that $\tau \neq 0$ and therefore that $\tau = u_1 \wedge \cdots \wedge u_k$ for some linearly independent vectors $u_1, \cdots, u_k$. By the definition of $\cG^T$, the operator $I+ t T$ preserves the subspace $\Span_{\overline{\FF}}\{ u_1,\ldots, u_k\}$.  This implies that $(I + t T)(\tau) \in \overline{\FF} \tau$.  But $T$ is nilpotent, so the only eigenvalue of $I+ t T$ acting on $\bigwedge^k V$ is $1$.  Hence $(I+ t T)(\tau) = \tau$ as claimed.
\end{proof}

\subsection{The proof of the KLMW Conjecture}
\label{section:proofKLMW}

As before, let us define
\begin{align}
  \label{eq:55}
V_{[-N,N)} = \Span \{ e_{-N}, e_{-N+1}, \cdots, e_{0}, \cdots, e_{N-1} \}  
\end{align}
and recall that we have:
\begin{align}
  \label{eq:51}
  \SGr  = \bigcup_{N\geq1} \Gr(N,V_{[-N,N)})  
\end{align}
For each $N$, let us consider the  operator $t^n : V_{[-N,N)} \rightarrow V_{[-N,N)}$ that shifts the index up by $n$. That is,
$
  t^n( e_i) = e_{i+n},
$
where we interpret $e_{i+n} = 0$ if $i+n \geq N$. We see that $t^n$ is a nilpotent operator, and we can consider
the scheme of $t^n$-invariant subspaces $\cG^{t^n}_N \subseteq \Gr(N,V_{[-N,N)})$. Let us write $\cS^{t^n}_N$ for the corresponding scheme given by the vanishing of the shuffle operators.

We have
$
  \Gr_{\GL_n}  = \bigcup_{N\geq1} \cG^{t^n}_N
$,
and 
$
  \VV(sh_\bullet^{(n)}) \cap \SGr=  \bigcup_{N\geq1} \cS^{t^n}_N
  $.
Therefore, by Theorem \ref{thm:FinAnalog}, we have:
\begin{align}
  \label{eq:56}
\VV(sh_\bullet^{(n)}) \cap \SGr \subseteq \Gr_{\GL_n}  
\end{align}
Recall that by Theorem \ref{thm:KLMW} $\VV(sh^{(n)}_\bullet)=\PP(V(\Lambda_0))$.  Hence by Proposition \ref{prop:simplified-klmw-problem}, this completes the proof of Theorem \ref{thm:klmw-conjecture}.


\newcommand{\Mayas}{\mathtt{Mayas}}
\section{Conjectures}

\subsection{The homogenous coordinate ring of $\Gr_{\SL_n}$}
\label{sec:conjGr}

Let $v=e_{i_1}\wedge e_{i_2}\wedge\cdots $ be a semi-infinite wedge vector written in normally ordered form, so that $i_1<i_2<\cdots$.  Recall that $v \in \cF$ if and only if $i_k=k-1$ for $k>>0$, and by definition these vectors form a basis of $\cF$.  The sequence of integers $\bm=(i_1,i_2,...)$ is the \textit{Maya diagram} associated to $v$ (cf. Section \ref{sec:Maya}).  Let $\Mayas$ be the set of (charge-$0$) Maya diagrams, i.e. the set of infinite increasing sequences $\bm=(i_1,i_2,...)$ such that $i_k=k-1$ for $k>>0$.  

To $\bm\in \Mayas$ we associate the Pl\"ucker coordinate $X_\bm$ in the homogenous coordinate ring of $\PP(\cF)$.  The Pl\"ucker embedding $\SGr \to \PP(\cF)$ identifies the homogenous coordinate ring $\Bbbk[\SGr]$ as a quotient of the polynomial ring $\Bbbk[X_\bm:\bm\in\Mayas]$:
\begin{align}
\Bbbk[\SGr]=\Bbbk[X_\bm:\bm\in\Mayas]/\cI
\end{align}
Here $\cI$ is the ideal generated by the Pl\"ucker relations
\begin{align}
  \label{eq:77}
X_{\bm}  X_{\bn} = \sum_{0 < t_1 < \cdots < t_d} X_{(i_1,..., j_1,..., j_d,...)}  X_{(i_{t_1},...,i_{t_d}, j_{d+1},j_{d+2},...)} 
\end{align}
where $\bm=(i_1,i_2,...),\bn=(j_1,j_2,...)\in\Mayas$.  Note that the above sum is finite.

We define the ideal $\cJ_n \subset \Bbbk[\SGr]$ to be generated by the $n$-shuffle equations:
\begin{align}
  \label{eq:78}
\cJ_n=\left\langle \lambda \circ \mathrm{sh}^{(n)}_d:\lambda \in \cF^*, d\geq1 \right\rangle
\end{align}
We can express $\cJ_n$ in coordinates in the following way.  Let $\varepsilon_i=(0,...,1,0,...)$ be an infinite sequence of zeroes except in the $i$-th position there is a $1$.  For $\bm\in\Mayas$ we set $X_{\bm-\varepsilon_i}=\pm X_{\bm'}$, where $\bm'$ is the Maya diagram obtained from $\bm$ by subtracting $1$ from the $i$-th entry, and reordering the entries in strictly increasing order if possible and keeping track of signs ($X_{\bm'}=0$ otherwise).  Then 
\begin{align}
  \label{eq:79}
\cJ_n=\left\langle \sum_{0 < t_1 < \cdots < t_d}X_{\bm-n(\varepsilon_{t_1}+\cdots+\varepsilon_{t_d})}:\bm\in\Mayas \right\rangle
\end{align}
By Theorem \ref{thm:klmw-conjecture} $\Gr_{\SL_n}$ is defined in $\SGr$ by $\cJ_n$.  However we do not know that $\cJ_n$ is saturated; if so then it would define the homogeneous coordinate ring $\Bbbk[\Gr_{\SL_n}]$.  Since $\Gr_{\SL_n}$ is reduced this is equivalent to the following conjecture.

\begin{Conjecture}
\label{conj:radical}
The ideal $\cJ_n \subset \Bbbk[\SGr]$ is radical, and hence 
$$
\Bbbk[\Gr_{\SL_n}]=\Bbbk[\SGr]/\cJ_n
$$
\end{Conjecture}

\subsection{The schemes  $\cS^T$ and $\cG^T$}
\label{sec:conjs-finite}

Recall that after reducing to the finite dimensional setting, our proof of the KLMW Conjecture follows from Theorem \ref{thm:FinAnalog} which says that $\cS^T$ is a closed subscheme of $\cG^T$.  
Motivated by the analogy in \S \ref{sec:studyingTinvs}, we will propose a finite dimensional analogue of the KLMW Conjecture.  First, we observe that $\cG^T$, like $\Gr_{\GL_n}$,  is not necessarily reduced.

\begin{example}
  \label{example:reduced}
  For simplicity, we suppose that $\Bbbk$ is a field. Consider $V = \Bbbk^n$ with its usual basis $\{ e_i \suchthat i \in [n] \}$, and let $T : V \rightarrow V$ be the operator given by $T e_1 = 0$ and $T e_i = e_{i-1}$ for $i \geq 2$. Fix $k \in [n]$. We then consider $\cG^T \subseteq \PP(\bigwedge^k V)$. it is clear that $\cG^T$ has only one closed point consisting of the $k$-dimensional subspace $\Span \{e_1, \cdots, e_k \}$, and this closed point has residue field $\Bbbk$ .

 Consider the dual numbers $R = \Bbbk[\eps]/\eps^2$. We see that the map $\cG^T(\Bbbk) \rightarrow \cG^T(R)$ is not surjective. In particular, the subspace $\Span \{ e_1, \cdots, e_{k-1}, e_k + \eps e_{k+1} \}$ is not in the image of this map. Therefore the tangent space at the unique closed point of $\cG^T$ is not trivial. Hence $\cG^T$ is not reduced.
\end{example}

We are therefore led to the following conjecture. 
\begin{Conjecture}
\label{conj:reduced}
Let $T : V \rightarrow V$ be a nilpotent operator. Then the scheme $\cG^T$ is not reduced unless $T=0$.
\end{Conjecture}

Now recall (Section \ref{sec:affgrass}) that $\Gr_{\SL_n}=(\Gr_{\GL_n}^{(0)})_{red}$ and Theorem \ref{thm:klmw-conjecture} says that $\Gr_{\SL_n}=\SGr\cap \PP(V(\Lambda_0))$.  The finite dimensional analogue of the KLMW Conjecture predicts a similar relationship between $\cS^T$ and $\cG^T$.  More precisely:

\begin{Conjecture}
\label{conj:finiteKLMW}
Let $T : V \rightarrow V$ be a nilpotent operator.  Then  diagram \eqref{diagram:cart2} is Cartesian.
Equivalently, $\cS^T$ is the reduced scheme of $\cG^T$.
\end{Conjecture}

Note that by Proposition \ref{prop:Fpoints} this conjecture is equivalent to the claim that $\cS^T$ is reduced.

\begin{example}
  Let $\Bbbk$ be a field, and let $V$ and $T$ be as in Example \ref{example:reduced}. Consider the usual basis $\{ e_I^* \suchthat I \subset [n], |I| = k \}$ of $(\bigwedge^k V)^*$. Unwinding the definition of $\cS^T$ in this case, we find that $\cS^T$ is a closed subscheme of $\cap_{I \neq \{1,\cdots,k\}} \VV(e_I^*) \subseteq \PP(\bigwedge^k V)$. Notice that $\cap_{I \neq \{1,\cdots,k\}} \VV(e_I^*)$ is isomorphic to $\Spec \Bbbk$. Because $\cS^T$ is non-empty in this case, we conclude that $\cS^T$ is isomorphic to $\Spec \Bbbk$. In particular, it is reduced.
\end{example}

Note, that if $T$ is not nilpotent, both of these conjectures fail. For example, if $T$ is the identity map then $\cG^T$ is the full Grassmannian, and $\cS^T$ is empty.  Looking at the proof of Proposition \ref{prop:Fpoints}, this failure stems from the fact that general $T$ have non-zero eigenvalues (so the vanishing of all $sh_d^T$ is no longer the correct condition to impose).

Finally, we expect also that an analogue of Conjecture \ref{conj:radical} holds in the finite dimensional setting.  More precisely, we define the {\it $T$-shuffle ideal}  $\cJ_T \subset \Bbbk[\Gr(k,V)]$ by: 
\begin{align}
  \label{eq:80}
\cJ_T =\left\langle \lambda\circ sh_d^T: \lambda \in \bigwedge^kV^*, d\geq 1 \right\rangle
\end{align}
\begin{Conjecture}
\label{conj:saturation}
The ideal $\cJ_T$ is saturated, and hence
$$
\Bbbk[\cS^T]\cong \Bbbk[\Gr(k,V)]/\cJ_T
$$
\end{Conjecture}

Note that Conjectures \ref{conj:finiteKLMW} and \ref{conj:saturation} are together equivalent to the claim that $\cJ_T$ is radical.
Computer experiments confirm that this is true for small examples and for small characteristics of the ground field.


\appendix
\newcommand{\regularize}[1]{\rho_{#1}}
\newcommand{\biggaplocation}[1]{\ell_{#1}}
\newcommand{\Parts}{\mathtt{Partitions}}
\newcommand{\Part}{\mathtt{Partition}}
\newcommand{\nregParts}[1]{#1\!-\!\mathtt{regPartitions}}
\renewcommand{\Maya}{\mathtt{Maya}}
\newcommand{\gdomeq}{\unrhd}
\newcommand{\ASL}{\widehat{\SL}}

\section{$n$-regular partitions and KLMW basis}
\label{appendix}
Kreiman, Lakshmibai, Magyar, and Weyman define a basis in the degree-one part of the homogeneous coordinate ring of $\Gr_{\SL_n}$.  In this appendix we recall their construction and exhibit some interesting combinatorial properties of this basis.   

\subsection{Maya diagrams}
\label{sec:Maya}

The set $\Mayas$ (cf. Section \ref{sec:conjGr}) is in bijection with $\Parts$, the set of partitions.  The bijection
\begin{align*}
\Part:\Mayas \to \Parts
\end{align*}
is given by $\Part(\bm) = \mu^T$ (the transpose partition), where:
\begin{align}
  \label{eq:81}
\mu=-(i_1,i_2-1, i_3-2,...)     
\end{align}
Note that because $i_{j} = j-1$ for sufficiently large $j$, almost all parts of $\mu$ are zero.  Let us write
\begin{align*}
  \Maya :  \Parts \rightarrow  \Mayas
\end{align*}
for the inverse bijection.  Note also that this is the transpose of the bijection used in Section \ref{section:Fock}.

Let $\bm = (i_1,i_2,\dotsc) \in \Mayas$. We visualize $\bm$ as follows: consider a set of slots indexed by the integers. We fill each slot corresponding to an element of  $\{i_1, i_2, \dotsc \}$ with a black bead, and the remaining empty slots are depicted by a white circle.
For example if $\bm = (-3,-1,0,3,4,5,\dotsc)$, then we have the following configuration:
$$
\cdots \hspace{5mm} \underset{\tiny{-5}}{\circ} \hspace{5mm} \underset{\tiny{-4}}{\circ} \hspace{5mm}\underset{\tiny{-3}}{\bullet} \hspace{5mm} \underset{\tiny{-2}}{\circ} \hspace{5mm} \underset{-1}{\bullet} \hspace{5mm}\underset{0}{\bullet} \hspace{5mm} \underset{1}{\circ} \hspace{5mm} \underset{2}{\circ} \hspace{5mm}\underset{3}{\bullet} \hspace{5mm} \underset{4}{\bullet} \hspace{5mm} \underset{5}{\bullet} \hspace{5mm} \cdots
$$
Just as we identify a partition with its Young diagram, we identify Maya diagrams with these pictorial representations; in fact, this pictorial representation is the reason for the terminology ``Maya diagram''.
For example, if we ignore signs, we can say that the shuffle operator ${sh^{(n)}_d}^*$ acts on a Maya diagram, by summing over all possible ways of choosing $d$ black beads and moving them each $n$ slots to the left (hence the name ``shuffle''). 

Finally, we write $\geq_{mlex}$ for the lexicographical order on Maya diagrams. We will also write $\geq_{mlex}$ for the corresponding order on $\Parts$ via the bijection $\Part : \Mayas \rightarrow \Parts$.

\subsection{The KLMW basis}
\label{appendix:2}
We have a basis of $\cF$ indexed by partitions: to $\lambda$ we associate the semi-infinite wedge $s_\lambda=e_{i_1}\wedge e_{i_2}\wedge \cdots$, where $\Maya(\lambda)=(i_1,i_2,...)$.

We have $\cF = \oplus_{\lambda \in \Parts} \Bbbk s_\lambda$. Therefore, $\cF^* = \prod_{\lambda \in \Parts} \Bbbk s_\lambda^*$ where
$\langle s_\lambda^*, s_\mu \rangle = \delta_{\lambda,\mu}$.
Recall that a partition is  {\it $n$-regular} if none of its non-zero parts appear $n$ or more times. Let us write $\nregParts{n}$ for the set of $n$-regular partitions. We write:
\begin{align}
  \label{eq:68}
\overline{\Span_\Bbbk}\{ s_\lambda^* \suchthat \lambda \in \nregParts{n} \} =  \prod_{\lambda \in \nregParts{n}} \Bbbk s_\lambda^* \subseteq \cF^*
\end{align}

\begin{Theorem}{\cite{KLMW}}
  The subspace
  \begin{align*}
    \overline{\Span_\Bbbk}\{ s_\lambda^* \suchthat \lambda \in \nregParts{n} \} \subset \cF^*
  \end{align*}
  maps isomorphically onto $V(\Lambda_0)^*$.
\end{Theorem}

Let $v_\lambda^*=s_\lambda^*|_{V(\Lambda_0)}$.  By the theorem  $\{v_\lambda^*  \suchthat\lambda \in \nregParts{n} \}$ is a (topological) basis of ${V(\Lambda_0)}^*$ which we call the {\it KLMW basis}.  We denote the corresponding dual basis (the {\it dual KLMW basis})  of $V(\Lambda_0)$ by $\{v_\lambda  \suchthat\lambda \in \nregParts{n} \}$.  In \cite{KLMW}, the authors give an explicit ``shuffle algorithm'' to write $v_\lambda^*$ for arbitrary $\lambda$ in terms of the KLMW basis, which we'll now recall. 

\subsection{The KLMW algorithm}

Fix $n \geq 2$. We will define two maps:
\begin{align*}
 \regularize{n} : \Parts \rightarrow \Parts
\end{align*}
\begin{align*}
 \biggaplocation{n} : \Parts \rightarrow \NN \sqcup \{\infty\}
\end{align*}
Let $\lambda \in \Parts$, and let $\bm=( i_1, i_2, \cdots) = \Maya(\lambda)$. Then we define $\biggaplocation{n}(\lambda)$ to be the least $\ell$ such that $i_{\ell} - i_{\ell+1} > n$.  If no such $\ell$ exists, we define $\biggaplocation{n}(\lambda)=\infty$. Note that this is precisely the case when $\lambda$ is an $n$-regular partition.
If $\biggaplocation{n}(\lambda) < \infty$, then set 
\begin{align}
  \label{eq:82}
\bm'=(i_1+n,i_2+n, \cdots, i_{\ell} +n, i_{\ell+1}, \cdots)
\end{align}
where $\ell=\biggaplocation{n}(\lambda)$.
 In this case, we define:
\begin{align}
 \regularize{n}(\lambda) = \Part(\bm') 
\end{align}
If $\biggaplocation{n}(\lambda) = \infty$, then we define $\regularize{n}(\lambda) = \lambda$.  

\begin{Proposition}{\cite[Theorem 4.6.3]{KLMW}}
\label{prop:KLMW}
Suppose $\lambda$ is a non-$n$-regular partition.  Let $d=\ell_n(\lambda)$ and let $\mu=\rho_n(\lambda)$.   Then 
\begin{align}
\label{eq:klmw-algorithm}
(sh^{(n)}_d)^*(v_\mu^*)= \pm v_\lambda^* + (\text{lower order terms}) 
\end{align}
with respect to the order $\geq_{mlex}$ on $\Parts$.
\end{Proposition}
We therefore have the following algorithm to express  $v_\lambda^*$, modulo $\fS$ (the shuffle equations), as a sum of basis vectors corresponding to $n$-regular partitions: If $\lambda$ is $n$-regular, then we are done. Otherwise, we apply \eqref{eq:klmw-algorithm}. In \cite{KLMW}, the authors show that $(sh^{(n)}_d)^*(v_\mu^*)$ lies in $\fS$. Therefore modulo $\fS$, 
$v_\lambda^*$ is equal to a sum of terms that are strictly lower with respect to $\geq_{mlex}$. By induction, all of these strictly lower terms are congruent to a sum of $n$-regular partitions modulo $\fS$.

The first result of this appendix refines Proposition \ref{prop:KLMW} by showing that the additional terms in (\ref{eq:klmw-algorithm}) are also lower with respect to the dominance order on partitions. In fact, we will show a stronger statement: the additional terms are lower with respect to a coarser order that we call the $n$-dominance order.

\subsection{Dominance order on Maya diagrams}

Let $\bm= ( i_1, i_2, \cdots) \in \Mayas$, with corresponding bead set $I_\bm = \{i_1, i_2, \cdots \}$.
Let $\alpha < \beta$ be two strictly positive integers, and suppose $i_\alpha + 1, i_\beta - 1 \notin I_\bm \backslash \{i_\alpha,i_\beta\}$ . In this case we define a new bead set $I_{\bm'}= (I_\bm \backslash \{i_\alpha,i_\beta\}) \cup \{i_\alpha + 1, i_\beta - 1 \}$, which is the bead set of a Maya diagram $\bm'$.  
We say that the operation $\bm \mapsto \bm'$ is given by a \textit{$1$-jump pair} $(i_\alpha,i_\beta) \mapsto (i_\alpha+1,i_\beta-1)$.
We define the {\it $1$-jump order} on $\Mayas$ to be the order generated by the relations $\bm \geq_{1-jump} \bm'$ where $\bm \mapsto \bm'$ is given by a $1$-jump pair. The following is immediate.

\begin{Lemma} 
  \label{lem:1-jump-to-dominance}
Under the isomorphism $\Part : \Mayas \rightarrow \Parts$, the $1$-jump order corresponds to the dominance order on partitions.
\end{Lemma}

Similarly, we can generalize the notion of jump pairs for any strictly positive integer $n$ as follows. Let $\bm$ and $I_\bm$ be as above. Let $\alpha < \beta$. If  $i_\alpha + n, i_\beta - n \notin I_\bm \backslash \{i_\alpha,i_\beta\}$ and $i_\alpha + n \leq i_\beta$, then we define 
$I_{\bm''}= (I_\bm \backslash \{i_\alpha,i_\beta\}) \cup \{i_\alpha + n, i_\beta - n \}$, which is the bead set of a Maya diagram $\bm''$.  
We say that the operation $\bm \mapsto \bm''$ is given by a \textit{$n$-jump pair} $(i_\alpha,i_\beta) \mapsto (i_\alpha+n,i_\beta-n)$.
We can define an analagous \textit{$n$-jump order}  $\geq_{n-jump}$ on $\Mayas$. The following is clear.
\begin{Lemma}
  \label{lem:n-jump-to-one-jump}
 Let $\bm,\bm' \in \Mayas$. If $\bm \geq_{n-jump} \bm'$, then $\bm \geq_{1-jump} \bm'$.
\end{Lemma}

\begin{Proposition}
  \label{prop:lower-in-dominance}
All the extra terms in (\ref{eq:klmw-algorithm}) are lower in the dominance order.
\end{Proposition}

\begin{proof}
 Assume $\lambda$ is not $n$-regular. Let $\ell = \biggaplocation{n}(\lambda)$, and let $\bm=\Maya(\lambda)= ( i_1, i_2, \cdots)$. 
Let $\bm'=\Maya(\regularize{n}(\lambda)) =(i_1+n, \cdots i_\ell+n,i_{\ell+1}, \cdots)$. The terms in  (\ref{eq:klmw-algorithm}) are precisely given by taking $\ell$ distinct beads of $m'$ and shifting them each $n$ slots left. 

When we shift precisely the left-most $\ell$ beads, we recover $\bm$.  Otherwise, consider $\ell$ distinct beads of $\bm'$ that can be each moved $n$ slots to the left, and let $\bm''$ be the resulting Maya diagram.  Let $k$ be the number of the chosen beads that are in the sublist $(i_{\ell+1}, \cdots)$ of $m$, and let $({\beta_1}, \cdots {\beta_k})$ be the {\it indices} of those $k$ beads. So the actual beads are $(i_{\beta_1},\cdots i_{\beta_k})$.

Now, $\ell - k$ of the chosen beads are in the sublist $(i_1+n, \cdots i_\ell+n)$. Let $(\alpha_k, \alpha_{k-1}, \cdots \alpha_1)$ be the indices (in increasing order) of the $k$ beads that are {\it not} chosen. So the actual beads are $(i_{\alpha_k}+n, i_{\alpha_{k-1}}+n, \cdots i_{\alpha_1}+n)$.

Then $\bm''$ is obtained from $\bm$ by
\begin{itemize}
 \item taking the beads $(i_{\alpha_k}, i_{\alpha_{k-1}}, \cdots, i_{\alpha_1})$ of $\bm$ and replacing them with $(i_{\alpha_k}+n, i_{\alpha_{k-1}}+n, \cdots, i_{\alpha_1}+n)$
 \item taking the beads $(i_{\beta_1},\cdots, i_{\beta_k})$ of $\bm$ and replacing them with $(i_{\beta_1}-n,\cdots, i_{\beta_k}-n)$
\end{itemize}

For all $j \in \{1, \cdots ,k \}$, $\alpha_j \leq \ell$ and $\beta_j \geq \ell +1$. The condition $i_{\ell+1} - i_\ell > n$ guarantees that $(i_{\alpha_1},i_{\beta_1}) \mapsto (i_{\alpha_1}+n,i_{\beta_1}-n)$ satisfies the condition of being an $n$-jump pair for the Maya diagram $\bm_0=\bm$. Let $\bm_0 \mapsto \bm_1$ be the operator generated by this $n$-jump pair. Inductively, we define $\bm_0 \mapsto \bm_1 \mapsto \cdots \mapsto \bm_k$, where the operation $\bm_{j-1} \mapsto \bm_j$ is generated by the $n$-jump pair $(i_{\alpha_j},i_{\beta_j}) \mapsto (i_{\alpha_j}+n,i_{\beta_j}-n)$. Observing that $\bm_k = \bm''$, we conclude that $\bm \geq_{n-jump} \bm''$. Applying Lemmas \ref{lem:1-jump-to-dominance} and \ref{lem:n-jump-to-one-jump}, we have our result.

\end{proof}

\subsubsection{The $n$-dominance order}

\begin{Definition}
 We define the {\it $n$-dominance order} on $\Parts$ to be the order corresponding to $\geq_{n-jump}$ under the bijection $\Part : \Mayas \rightarrow \Parts$. 
\end{Definition}

The proof of Proposition \ref{prop:lower-in-dominance} implies the following stronger statement.

\begin{Proposition}
  The extra terms in \eqref{eq:klmw-algorithm} are lower than $\lambda$ in the $n$-dominance order.
\end{Proposition}

It is immediate to check that a pair of Maya diagrams that differ by an $n$-jump pair have the same $\ASL_n$ weight. Therefore, two partitions can be in the same connected component of the $n$-dominance poset only if they have the same $\ASL_n$ weight. Based on computer experiments, we conjecture that the converse is true.

\begin{Conjecture}
  If $\lambda, \lambda' \in \Parts$ have the same $\ASL_n$ weight, then they lie in the same connected component of the $n$-dominance poset.
\end{Conjecture}

\subsection{The KLMW basis via Kostka-Foulkes matrices and specializations}

\renewcommand{\v}[1]{v_#1}
\newcommand{\nRegParts}[1]{#1\mathtt{RegularPartitions}}
For every $n$-regular partition $\lambda$, we can expand the dual KLMW basis element $v_\lambda$ in $\cF$ 
\begin{align}
  \v{\lambda} = \sum_\nu d_{\lambda,\nu} s_\nu
\end{align}
where $\nu$ varies over all partitions. By the KLMW algorithm, we know  $d_{\lambda,\mu}\in \ZZ$. By definition, $d_{\lambda,\nu} = \delta_{\lambda,\nu}$ when $\nu$ is also $n$-regular, so the interesting coefficients occur when $\nu$ is not $n$-regular.

Let $K(t) = \left( K(t) \right)_{\mu,\lambda}$ be the Kostka-Foulkes matrix (see \cite[Chapter III]{M} for our conventions). Let $C(t) = K(t)^{-1}$.
Then we can let $A(t)$ be the inverse of the submatrix of $C(t)$ whose rows and columns are indexed by $n$-regular partitions. This inverse exists because $C(t)$ is a unitriangular matrix. Let $B(t)$ be the submatrix of $C(t)$ whose rows are indexed by $n$-regular partitions and whose columns are indexed by all partitions. Then define the matrix 
\begin{align}
  D(t) = A(t)B(t)
\end{align}
Let $\zeta$ be a primitive $n$-th root of unity. Then we have the following theorem. 

\begin{Proposition}
\label{prop:KF}
  \begin{align*}
    D(\zeta)_{\lambda,\mu} = d_{\lambda,\mu}
  \end{align*}
\end{Proposition}

\begin{proof}
The Hall-Littlewood $P$-functions are related to the Schur functions by: 
\begin{align*}
 P_{\lambda}(x;t) = \sum_{\mu} C(t)_{\lambda,\mu} s_\mu 
\end{align*}
It is also known, that $\{ P_\lambda(x;\zeta) | \lambda \in\nregParts{n} \}$
is a basis $V(\Lambda_0)$ over the ring $\QQ(\zeta)$. (see \cite[III.7 Example 7]{M}). Therefore we can write:
\begin{align}
  \label{eq:v-lambda}
\v{\lambda} = \sum_{\mu \in \nregParts{n}} a_{\lambda,\mu} P_{\mu}(x;\zeta)
\end{align}
Pairing both sides of \eqref{eq:v-lambda} with $s_\nu$ for $\nu \in \nregParts{n}$, we get:
\begin{align*}
\delta_{\lambda,\nu} = \sum_{\mu \in \nregParts{n}} a_{\lambda,\mu} C(\zeta)_{\mu,\nu}
\end{align*}
That is, we see that
$
 a_{\lambda,\mu} = A(\zeta)_{\lambda,\mu} 
 $.
Pairing both sides of \eqref{eq:v-lambda} with arbitrary $s_\nu$, we get 
\begin{align*}
d_{\lambda,\nu} = \sum_{\mu \in \nregParts{n}} a_{\lambda,\mu} C(\zeta)_{\mu,\nu} 
\end{align*}
i.e.,
$
 d_{\lambda,\nu} = (A(\zeta) \cdot B(\zeta))_{\lambda,\nu}
 $.
\end{proof}

This gives another proof that the dual KLMW basis is unitriangular with respect to the Schur functions under the dominance order.



\end{document}